\documentclass[11pt]{article}

\usepackage[T1]{fontenc}
\usepackage[utf8]{inputenc}
\usepackage{amsmath,amsthm,amssymb,mathtools}
\usepackage{bm}
\usepackage{xcolor}
\usepackage{enumitem}
\usepackage{hyperref}
\hypersetup{
  hidelinks,
  pdftitle={Dry-Friction Inertial Dynamics with Implicit Hessian-Driven Damping},
  pdfauthor={Samir Adly},
  pdfsubject={Finite-time stabilization, shadowing, and proximal discretization},
  pdfkeywords={dry friction, implicit Hessian-driven damping, shifted-gradient force,
    inertial dynamics, finite-time stabilization, proximal algorithm}
}
\usepackage{microtype}
\usepackage[a4paper,margin=0.8in]{geometry}
\usepackage{subcaption}
\usepackage{graphicx}
\usepackage{float}
\usepackage{placeins}
\usepackage{booktabs}
\usepackage{array}
\usepackage{tabularx}

\newcolumntype{Y}{>{\raggedright\arraybackslash}X}

\theoremstyle{plain}
\newtheorem{theorem}{Theorem}[section]
\newtheorem{proposition}[theorem]{Proposition}
\newtheorem{lemma}[theorem]{Lemma}
\newtheorem{corollary}[theorem]{Corollary}
\theoremstyle{definition}
\newtheorem{definition}[theorem]{Definition}

\newtheorem{remark}[theorem]{Remark}
\newtheorem{assumption}[theorem]{Assumption}

\newcommand{\R}{\mathbb{R}}
\newcommand{\N}{\mathbb{N}}
\newcommand{\B}{\mathbb{B}}
\newcommand{\Hc}{\mathcal H}
\newcommand{\E}{\mathcal E}
\newcommand{\Id}{\mathrm{Id}}
\newcommand{\Dom}{\mathrm{Dom}\,}
\newcommand{\dom}{\mathrm{dom}\,}
\newcommand{\argmin}{\mathop{\mathrm{arg\,min}}}
\newcommand{\prox}{\mathrm{prox}}
\newcommand{\dist}{\mathrm{dist}}
\newcommand{\bd}{\mathrm{bd}}
\newcommand{\interior}{\mathrm{int}}
\newcommand{\ip}[2]{\langle #1,#2\rangle}
\newcommand{\norm}[1]{\left\lVert #1\right\rVert}
\newcommand{\GammaZero}{\Gamma_{0}(\Hc)}
\newcommand{\Cfr}{C_{\phi}}
\newcommand{\Tstop}{T_{\mathrm{stop}}}
\newcommand{\Kstop}{K_{\mathrm{stop}}}

\definecolor{revisionorange}{RGB}{220,100,0}

\AtBeginDocument{
  \setlength{\abovedisplayskip}{4pt plus 2pt minus 2pt}
  \setlength{\belowdisplayskip}{4pt plus 2pt minus 2pt}
  \setlength{\abovedisplayshortskip}{2pt plus 1pt}
  \setlength{\belowdisplayshortskip}{2pt plus 1pt}
}

\makeatletter
\renewenvironment{abstract}
  {\par\smallskip\small\noindent\textbf{Abstract. }\ignorespaces}
  {\par\medskip}
\makeatother

\title{{Dry-Friction Inertial Dynamics with Implicit Hessian-Driven
Damping: Finite-Time Stabilization, Shadowing, and Proximal Discretization}}
\author{
Samir Adly\thanks{Laboratoire XLIM, Universit\'e de Limoges, 123 avenue
Albert Thomas, 87060 Limoges CEDEX, France. Email:
\texttt{samir.adly@unilim.fr},
\url{https://www.unilim.fr/pages_perso/samir.adly/}.}
}
\date{}

\begin{document}
\maketitle

\begin{abstract}
In a real Hilbert space $\Hc$, we study the following inertial differential inclusion
$$
\ddot x(t)+\gamma\dot x(t)+\partial\phi(\dot x(t))
+\nabla f(x(t)+\beta\dot x(t))\ni0,
$$
where $\gamma>0$ is the viscous damping coefficient, and $\phi$ is a convex potential with a sharp minimum at the origin that models the dry friction damping (typically $\phi=r\Vert\cdot\Vert$ where $r>0$ is the dry-friction parameter). The function $f$ represents the smooth potential to be minimized, and the shifted-gradient evaluation 
$\nabla f(x(t)+\beta\dot x(t))$ is known as the implicit Hessian-driven model. Here $\beta\geq 0$ represents the corresponding Hessian-driven parameter.  
Both the explicit and the implicit Hessian-driven dynamics are know to attenuate the oscillations that occurs in inertial systems. Our contribution concerns the quantitative analysis of this continuous dynamic and its temporal discretization counter-part under the action of these three combined dampings: viscous damping, dry friction, and implicit Hessian-driven damping. We establish a global well-posedness, an exact Lyapunov analysis adapted to the implicit Hessian-driven damping, finite length of the trajectory and its strong convergence to an approximate critical point $x_\infty$ of $f$ satisfying: $-\nabla f(x_\infty)\in\partial \phi(0)$. 
We show finite-time stabilization under a strict interior condition on the terminal force. 
We also compare the explicit and the implicit Hessian-driven dynamics and show that their trajectories differ by $O(\beta^2)$ on finite horizons. A temporal semi-implicit discretization of the dynamic above leads to a proximal implicit Hessian-driven algorithm based on one shifted-gradient evaluation. 
 We derive its discrete Lyapunov analysis, asymptotic convergence and, under a strict terminal margin, finite convergence of the discrete iterates. 
 We also establish an $O(\sqrt h)$ continuous-discrete estimate and show that the implicit and the explicit Hessian algorithms differ by $O(\beta^2+\beta h)$ over finite physical horizons.
 Finally, an exact quadratic modal analysis identifies the Schur stability boundary and quantifies the gap with the nonlinear Lyapunov analysis. Numerical experiments illustrate each of these distinct conclusions.
\end{abstract}



\tableofcontents

% ============================================================================
\section{Introduction}

Second-order dissipative systems provide a natural link between
optimization algorithms and mechanical models.  Inertial dynamics involving viscous damping friction generally produces asymptotic decay and may generate oscillations. In contrast, dry
friction is modeled by a nonsmooth velocity potential $\phi$ satisfying
$0\in\operatorname{int}(\partial\phi(0))$. This condition introduces a
threshold. Under suitable assumptions, the velocity vanishes in finite time
and remains at zero. Continuous dynamics and inertial proximal
algorithms based on dry friction have been studied in
\cite{AdlyAttouchCabot2006,AdlyAttouch2020,AdlyAttouch2021,AdlyAttouch2022,AdlyAttouch2023,AdlyAttouchLe2024}.

Hessian-driven damping provides another way to reduce the
oscillations known to occur in inertial systems. It may be viewed as a geometrical damping governed by the Hessian operator of $f$. In its explicit form, it contains the term
$\beta\nabla^2f(x(t))\dot x(t)=\beta\frac{d}{dt}\nabla f(x(t))$ along a
sufficiently smooth trajectory. This approach was introduced in
\cite{AlvarezAttouchBolteRedont2002}. Its combination with dry friction was
studied in the continuous and discrete settings in
\cite{AdlyAttouch2020,AdlyAttouch2021}. The dynamics
$$
\ddot x(t)+\gamma(t)\dot x(t)
+\nabla f\bigl(x(t)+\beta(t)\dot x(t)\bigr)=0
$$
was introduced by Alecsa, L\'aszl\'o, and Pin\c{t}a
\cite{AlecsaLaszloPinta2021}, who studied the continuous system and the
inertial algorithms obtained by temporal discretization. This system is
now called the Inertial System with Implicit Hessian Damping (ISIHD). Later
works consider Tikhonov regularization \cite{Laszlo2025}, stochastic and
nonsmooth perturbations \cite{MaulenSotoFadiliAttouchOchs2026}, and strongly
quasiconvex functions in finite dimensional settings \cite{GaoSunHe2026}.

Following this terminology, we call the shifted-gradient term implicit
Hessian-driven damping. We study its interaction with dry friction when the
viscous and shift coefficients are constant. We are mainly interested in
finite-time stabilization, comparison with explicit Hessian-driven damping~\cite{AdlyAttouch2020},
and proximal discretization with finite convergence of the discrete
iterates. When $\phi\equiv0$ and $\beta=0$, the system reduces to the heavy
ball with friction.

%%%%%%%%%%%%%%%%%
The present paper continues the work \cite{AdlyAttouch2020}, where dry
friction was combined with explicit Hessian-driven damping in continuous and
discrete models. It is then natural to study the implicit case and to compare
the two approaches. Another related contribution is
\cite{AdlyAttouch2023}. There, time scaling and averaging lead to accelerated
dynamics with dry friction, asymptotically vanishing viscous damping, and
implicit Hessian-driven damping for convex optimization. Here the viscous
and shift coefficients are constant. The continuous analysis does not
require convexity of $f$.

The stochastic system studied in
\cite{MaulenSotoFadiliAttouchOchs2026} has time-dependent coefficients and a
convex objective. It also includes stochastic perturbations. In \cite{GaoSunHe2026}, the authors considered the autonomous shifted-gradient equation
without dry friction for strongly quasiconvex functions. Dry friction and Hessian-driven damping by means of split monotone inclusions was studied in \cite{TangHuang2026}. In the present paper, we compare the implicit
and explicit systems before and after permanent stopping. We also compare
their stopping times and stopping positions.
%%%%%%%%%%%%%%%%%%%%%%%%%

The ISIHD model itself is not new. Our contributions concern its
combination with dry friction. We prove an energy
estimate and characterize the limiting equilibrium. A strict interior
condition gives finite-time stabilization, while a scalar example shows
what may occur when the limiting force lies on the boundary. We compare the
implicit and explicit systems with an optimal estimate of order
$O(\beta^2)$. We then discretize the implicit system and study the limit as
$h\downarrow0$. Finally, we compare the sufficient step-size condition
obtained from the nonlinear Lyapunov analysis with the exact Schur stability
condition for quadratic functions.

Our first result concerns the implicit Hessian-driven inclusion with
dry friction. We do not assume that $f$ is convex or that its gradient is
globally Lipschitz continuous. We assume that $f$ is bounded from below,
that its gradient is Lipschitz continuous on bounded sets, and that
$\phi$ is proper, lower semicontinuous, and convex, with
$r\B\subseteq\partial\phi(0)$ for some $r>0$, where $\B:=\{u\in\Hc:\|u\|\leq1\}$ denotes the closed
unit ball of $\Hc$. An energy written in terms of
$y=x+\beta\dot x$ allows us to prove global well-posedness, finite length,
and strong convergence to a point $x_\infty$ satisfying
$$
-\nabla f(x_\infty)\in\partial\phi(0).
$$
If $-\nabla f(x_\infty)$ belongs to the interior of
$\partial\phi(0)$, the trajectory becomes constant in finite time. We give a
scalar example showing that the same conclusion may fail when
$-\nabla f(x_\infty)$ lies on the boundary.

We next compare the implicit and explicit Hessian-driven models. If the
Hessian is locally Lipschitz continuous, their trajectories differ by
$O(\beta^2)$ on every fixed time interval. We then obtain estimates of order $O(\beta^2)$ for
the two stopping times and stopping positions, and the trajectory estimate
holds on $[0,+\infty[$. The two systems coincide for quadratic functions. A
one-dimensional nonquadratic example shows that the order $O(\beta^2)$ is
optimal under the assumptions considered here.

Our third result concerns the discretization of the continuous dynamic. We treat inertia, viscous
damping, and dry friction semi-implicitly, and evaluate the gradient at the
shifted point. The resulting proximal algorithm requires one gradient
evaluation and one proximal evaluation at each iteration. We derive a
discrete energy and determine when its quadratic form is positive
semidefinite. Under this step-size condition, the sequence has finite
length and converges to a point satisfying the same equilibrium condition as
the continuous system. A strict interior condition gives finite convergence
of the discrete iterates. When $f$ is strongly convex, the equilibrium
condition also gives explicit estimates of the distance and the objective
gap to the unique minimizer of $f$.

We also study the approximation as $h\downarrow0$. The piecewise
affine interpolants differ from the continuous trajectory by
$O(\sqrt h)$ on bounded time intervals. Under a strict interior condition,
the physical stopping times $hK_{\mathrm{stop}}^h$ and the stopping
positions converge with the same order, and the trajectory estimate holds
on $[0,+\infty[$. The implicit and explicit algorithms differ by
$O(\beta^2+\beta h)$ on bounded physical time intervals and coincide for
quadratic functions. For quadratic functions without dry friction, a modal
analysis gives the exact Schur stability condition and shows that the
sufficient condition obtained from the nonlinear Lyapunov analysis is
conservative.

The numerical experiments illustrate these results. We first consider
the quadratic identity and the exact modal boundary. We then study the
damping effect for nonlinear functions. The last example concerns the case
where the limiting force lies on the boundary of $\partial\phi(0)$.

The arguments used here differ from those developed for dry-friction systems without Hessian damping \cite{AdlyAttouchCabot2006,AdlyAttouch2022} and for systems with explicit Hessian-driven damping \cite{AdlyAttouch2020,AdlyAttouch2021}. Without Hessian damping, the analysis relies on the usual mechanical energy written in terms of $x$ and $\dot x$. In the explicit case, the Hessian term can be treated through the identity
$\nabla^2 f(x)\dot x=\frac{d}{dt}\nabla f(x)$. This argument does not apply to the force $\nabla f(x+\beta\dot x)$. The present analysis instead uses the shifted variable $y=x+\beta\dot x$ and an energy containing both $f(y)$ and $\beta\phi(\dot x)$. The comparison with the explicit system then requires a second-order Taylor estimate, while the discretization leads to a different energy and a new step-size analysis.

The article is organized as follows. Section~\ref{sec:continuous-analysis}
studies the continuous implicit Hessian-driven system and its finite-time
stabilization. Section~\ref{sec:finite-horizon-comparison} compares the
implicit and explicit models. Section~\ref{sec:discrete-analysis} introduces
the proximal algorithm and proves its convergence properties.
Section~\ref{sec:consistance-disc} studies the limit as
$h\downarrow0$, the comparison of the two discretizations, and quadratic
stability. Section~\ref{sec:numerical-validation} contains the numerical
experiments. The appendices give the proof of the optimality result and the
algebraic verification of the discrete Lyapunov matrix.

% ============================================================================
\section{{Continuous Implicit Hessian-Driven Dynamics}}
\label{sec:continuous-analysis}

\subsection{{Models and Strong Solutions}}

Let $\Hc$ be a real Hilbert space with inner product $\ip{\cdot}{\cdot}$ and
norm $\norm{\cdot}$. Let $f:\Hc\to\R$ be differentiable and bounded from
below, and let $\phi:\Hc\to\R\cup\{+\infty\}$ be a proper lower
semicontinuous convex function. We fix $\gamma>0$, $\beta\geq0$, and initial
data
$$
x_0\in\Hc,
\qquad
v_0\in\Dom(\partial\phi).
$$

The three dynamics compared in this work are the following Cauchy problems.
The inclusion in each system is understood for almost every
$t\in[0,+\infty[$.
\begin{align}
\tag{DF}\label{eq:base-dynamics}
&\left\{
\begin{gathered}
\ddot x(t)+\gamma\dot x(t)+\partial\phi(\dot x(t))
+\nabla f(x(t))\ni0,\quad t\in[0,+\infty[,\\
x(0)=x_0,\qquad \dot x(0)=v_0;
\end{gathered}
\right.\\[0.4em]
\tag{EH$_\beta$}\label{eq:explicit-dynamics}
&\left\{
\begin{gathered}
\ddot x(t)+\gamma\dot x(t)+\partial\phi(\dot x(t))
+\beta\nabla^2 f(x(t))\dot x(t)+\nabla f(x(t))\ni0,
\quad t\in[0,+\infty[,\\
x(0)=x_0,\qquad \dot x(0)=v_0;
\end{gathered}
\right.\\[0.4em]
\tag{IH$_\beta$}\label{eq:implicit-dynamics}
&\left\{
\begin{gathered}
\ddot x(t)+\gamma\dot x(t)+\partial\phi(\dot x(t))
+\nabla f(x(t)+\beta\dot x(t))\ni0,
\quad t\in[0,+\infty[,\\
x(0)=x_0,\qquad \dot x(0)=v_0.
\end{gathered}
\right.
\end{align}
System \eqref{eq:base-dynamics} has no Hessian correction and was studied
in~\cite{AdlyAttouch2022}. System \eqref{eq:explicit-dynamics}, studied
in~\cite{AdlyAttouch2020}, contains the explicit Hessian-driven force.

System \eqref{eq:implicit-dynamics} is the implicit Hessian-driven
system with dry friction and constant coefficients. The system without dry
friction was introduced in \cite{AlecsaLaszloPinta2021}. The case with dry
friction and time-dependent coefficients arises from time scaling and
averaging in \cite{AdlyAttouch2023}. The autonomous smooth equation without
dry friction is studied for strongly quasiconvex functions in
\cite{GaoSunHe2026}. Here we analyze \eqref{eq:implicit-dynamics} directly
and compare it with \eqref{eq:explicit-dynamics}. The gradient in
\eqref{eq:implicit-dynamics} is evaluated at
$$
y(t)=x(t)+\beta\dot x(t).
$$
The force $\nabla f(x+\beta\dot x)$ and the explicit Hessian-driven
force are different. Taylor's formula gives their difference, and they
coincide when $f$ is quadratic. This comparison is carried out in
Section~\ref{sec:finite-horizon-comparison}. The word implicit refers to the
Hessian term obtained by expanding the gradient at the shifted point. It
does not mean that the evaluation of the gradient requires an implicit
nonlinear equation.

The continuous analysis is carried out under the following assumptions.

\begin{assumption}
\label{ass:basic}
The function $f$ belongs to $C^1(\Hc)$, its gradient is Lipschitz continuous
on bounded subsets of $\Hc$, and $\inf_{\Hc} f> -\infty$. The function
$\phi$ belongs to $\GammaZero$, satisfies $\phi(0)=\min\phi=0$, and there
exists $r>0$ such that
$$
r\B\subseteq\partial\phi(0).
$$
\end{assumption}

The following lemma gives an equivalent form of the sharp minimum
condition.

\begin{lemma}
Set $\Cfr=\partial\phi(0)$. Under the standing Assumption~\ref{ass:basic} on $\phi$,
$$
r\B\subseteq\Cfr
\quad\Longleftrightarrow\quad
\phi(v)\ge r\norm{v}
\quad\text{for every }v\in\Hc.
$$
\end{lemma}

\begin{proof}
Suppose first that $r\B\subseteq\partial\phi(0)$. For $v\ne0$, apply the
subgradient inequality at the origin to $r v/\norm{v}$. This gives
$\phi(v)\geq r\norm{v}$. The inequality is immediate at $v=0$.

Conversely, let $p\in r\B$. Then
$$
\ip{p}{v}\leq r\norm{v}\leq\phi(v)-\phi(0)
$$
for every $v\in\Hc$. Hence $p\in\partial\phi(0)$, which proves the reverse
inclusion.
\end{proof}

\begin{remark}
{When $\phi(v)=r\norm{v}$, we use throughout the paper the expression
\emph{simplified Coulomb friction}. From a mechanical viewpoint, this law
should not be identified with the full Coulomb law for unilateral contact,
which involves the normal contact reaction and, in its usual static-dynamic
formulation, distinguishes two friction coefficients. The potential
$\phi(v)=r\norm{v}$ retains the static threshold at zero velocity. Its
convex structure also gives a direct connection with proximal optimization
algorithms, which is the main reason for adopting this formulation here.}
\end{remark}

We first define the notion of strong solution used below.

\begin{definition}
Let $T>0$. A strong solution of \eqref{eq:implicit-dynamics} on $[0,T]$ is a
pair $(x,v)$ such that $x,v\in W^{1,2}(0,T;\Hc)$, $v=\dot x$ almost
everywhere, $x(0)=x_0$, $v(0)=v_0$, and there exists
$\xi\in L^2(0,T;\Hc)$ satisfying
$$
\xi(t)\in\partial\phi(v(t))
\quad\text{and}\quad
\dot v(t)+\gamma v(t)+\xi(t)+\nabla f(x(t)+\beta v(t))=0,
$$
for almost every $t\in(0,T)$.
\end{definition}

Strong solutions of \eqref{eq:base-dynamics} and
\eqref{eq:explicit-dynamics} are defined in the same way, with the
corresponding smooth force.

We distinguish permanent stopping from an isolated zero of the
velocity.

\begin{definition}
For a global strong solution, define
$$
\Tstop=\inf\left\{T\ge0:\ v(t)=0\text{ for every }t\ge T\right\},
$$
with the convention $\inf\varnothing=+\infty$.
\end{definition}

\subsection{Well-Posedness and Energy Dissipation}

On the phase space $\Hc\times\Hc$, introduce the two operators $\mathcal A:\,\Hc\times\Hc\rightrightarrows \Hc\times\Hc$ and $\mathcal B:\,\Hc\times\Hc\to \Hc\times\Hc$ defined respectively by
$$
\mathcal A(x,v)=\{0\}\times\partial\phi(v)
\quad
\mbox{ and}
\quad
\mathcal B(x,v)=\bigl(-v,\gamma v+\nabla f(x+\beta v)\bigr).
$$
The operator $\mathcal A$ is maximally monotone, while $\mathcal B$ is
Lipschitz continuous on bounded subsets of the phase space.

We use the following precise form of the Lipschitz perturbation theorem for
maximally monotone evolution equations. If $A$ is maximally monotone on a Hilbert
space, $F$ is globally Lipschitz continuous, and $z_0\in\Dom(A)$, then
\[
\dot z(t)+A(z(t))\ni-F(z(t)),\qquad z(0)=z_0,
\]
has a unique global strong solution; on every bounded time interval its
derivative and the corresponding selection from ${A(z)}$ are essentially
bounded. This is the standard strong Cauchy theory for a Lipschitz
perturbation of a maximally monotone operator; see
\cite[Chapter~III]{Brezis1973} and \cite[Chapter~4]{Barbu2010}.

We consider the energy
\begin{equation}
\E_{\beta}(t)
=f(y(t))-\inf_{\Hc}f
+\frac{1+\beta\gamma}{2}\norm{v(t)}^2
+\beta\phi(v(t)).
\label{eq:continuous-energy}
\end{equation}
When $\beta=0$, the last term is omitted rather than interpreted as
$0\cdot(+\infty)$.

The next theorem gives global well-posedness and the energy estimate.
\begin{theorem}
\label{thm:well-posedness}
\label{thm:continuous-energy}
Under Assumption~\ref{ass:basic}, for every
$(x_0,v_0)\in\Hc\times\Dom(\partial\phi)$, problem
\eqref{eq:implicit-dynamics} has a unique global strong solution. More
precisely, on every bounded interval $[0,T]$,
$$
x\in C^1([0,T];\Hc),
\qquad
v=\dot x\in W^{1,\infty}(0,T;\Hc),
$$
and the equation admits a selection
$\xi\in L^\infty(0,T;\Hc)$ with
$\xi(t)\in\partial\phi(v(t))$ almost everywhere. For this selection,
\begin{equation}
\dot\E_{\beta}(t)
+\beta\norm{\dot v(t)}^2
+\gamma\norm{v(t)}^2
+\ip{\xi(t)}{v(t)}=0
\label{eq:continuous-dissipation}
\end{equation}
for almost every $t>0$. Consequently, for every $T>0$,
\begin{align}
\label{eq:cont-integ-dissipation}
\E_\beta(T)
&+\beta\int_0^T\norm{\dot v(t)}^2\,dt
+\gamma\int_0^T\norm{v(t)}^2\,dt
+r\int_0^T\norm{v(t)}\,dt
\leq \E_\beta(0),
\end{align}
and
$$
\int_0^{+\infty}\norm{v(t)}\,dt
\leq \frac{\E_{\beta}(0)}{r}.
$$
When $\beta>0$, one also has
$\dot v\in L^2(0,+\infty;\Hc)$. If $\phi$ is finite and continuous on
$\Hc$, all the conclusions hold for every $v_0\in\Hc$.
\end{theorem}

\begin{proof}
Set $z=(x,v)$, $z_0=(x_0,v_0)$, and write
$$
\dot z+{\mathcal A(z)}+{\mathcal B(z)}\ni0,
\qquad z(0)=z_0.
$$
For $R>\norm{z_0}$, let $\Pi_R$ be the metric projection onto the closed
ball of radius $R$ in $\Hc\times\Hc$, and set
$\mathcal B_R=\mathcal B\circ\Pi_R$. Since $\Pi_R$ is nonexpansive and
$\mathcal B$ is Lipschitz continuous on bounded sets, $\mathcal B_R$ is
globally Lipschitz continuous. The Lipschitz perturbation theorem for
maximally monotone evolution equations therefore gives a unique global strong
solution $z_R=(x_R,v_R)$ of the truncated problem, with the stated local
regularity and an $L^\infty$ selection
$\xi_R(t)\in\partial\phi(v_R(t))$. Indeed,
$\mathcal A$ is maximally monotone because it is the Cartesian product of the
zero operator and the maximally monotone operator $\partial\phi$, and the
preceding perturbation theorem applies to $\mathcal A+\mathcal B_R$.

Let
$$
\tau_R=\inf\{t\geq0:\ \norm{z_R(t)}\geq R\}.
$$
On $[0,\tau_R)$, the truncated solution solves the original equation. Set
$y_R=x_R+\beta v_R$. Here the two chain rules are used in
their strong form. Since $y_R\in W^{1,\infty}$ and $f\in C^1$, the function
$f\circ y_R$ is absolutely continuous and its derivative is the classical
one. Moreover, if $u\in W^{1,2}(0,T;\Hc)$ and
$\eta\in L^2(0,T;\Hc)$ satisfy $\eta(t)\in\partial\phi(u(t))$ almost
everywhere, then $\phi\circ u$ is absolutely continuous and
$(\phi\circ u)'=\langle\eta,\dot u\rangle$ almost everywhere; see the convex
chain rule in \cite[Chapter~III]{Brezis1973}. Consequently,
$$
\frac{d}{dt}f(y_R)=\ip{\nabla f(y_R)}{v_R+\beta\dot v_R},
\qquad
\frac{d}{dt}\phi(v_R)=\ip{\xi_R}{\dot v_R}
$$
almost everywhere. Taking the inner product of the velocity equation with
$v_R+\beta\dot v_R$ and collecting exact derivatives yields
\eqref{eq:continuous-dissipation} on this interval. Since
$\ip{\xi_R}{v_R}\geq0$, the energy is nonincreasing. With
$$
V_0=\left(\frac{2\E_\beta(0)}{1+\beta\gamma}\right)^{1/2},
$$
one obtains
$$
\norm{v_R(t)}\leq V_0,
\qquad
\norm{x_R(t)}\leq\norm{x_0}+TV_0
$$
for $0\leq t\leq\min\{T,\tau_R\}$. These bounds are independent of $R$.
Choose
\[
R_T>\left[\bigl(\norm{x_0}+TV_0\bigr)^2+V_0^2\right]^{1/2}.
\]
Then every truncation radius $R\geq R_T$ satisfies
$\tau_R>T$, and $z_R$ solves the original equation on $[0,T]$. If
$R,S\geq R_T$, uniqueness for the globally Lipschitz truncated equations,
applied on the ball on which both truncations agree with $\mathcal B$, shows
that $z_R=z_S$ on $[0,T]$. Thus these solutions are compatible as $T$
increases and define a global strong solution of
\eqref{eq:implicit-dynamics}.

Two solutions on a bounded time interval remain in a common bounded set on
which $\mathcal B$ is Lipschitz continuous. Monotonicity of $\mathcal A$
and Gronwall's inequality give uniqueness. The regularity of $v$ and the
identity $x(t)=x_0+\int_0^t v(s)\,ds$ give the remaining regularity
assertions.

It remains to integrate the energy identity. The subgradient inequality at
zero and $r\B\subseteq\partial\phi(0)$ imply
$$
\ip{\xi(t)}{v(t)}\geq\phi(v(t))\geq r\norm{v(t)}.
$$
Integration of \eqref{eq:continuous-dissipation} gives
\eqref{eq:cont-integ-dissipation}. The energy is nonnegative, so
letting $T\to+\infty$ proves the finite-length estimate and, when
$\beta>0$, the $L^2$ estimate on $\dot v$.

Finally, if $\phi$ is finite and continuous on $\Hc$, then
$\Dom(\partial\phi)=\Hc$, and the restriction on $v_0$ disappears.
\end{proof}

The restriction on the initial velocity is relevant for an extended-valued
friction potential.

\begin{remark}
For a general $\phi\in\GammaZero$, the condition
$v_0\in\Dom(\partial\phi)$ is essential for the strong solution asserted
above. Initial data merely in $\overline{\dom(\phi)}$ lead to a weaker
Cauchy theory, with possible loss of regularity at $t=0$, which is not used
here.
\end{remark}

The dissipation estimate controls the acceleration only when the shift is
positive.

\begin{remark}
The energy controls the instantaneous value $\beta\phi(v(t))$, not the total
variation of $t\mapsto\phi(v(t))$. The direct acceleration estimate is
available only for $\beta>0$.
\end{remark}

Finite length gives the strong convergence of $x(t)$. We next identify
the limit without using finite-dimensional compactness.

% ============================================================================
\subsection{Limit Equilibrium and Finite-Time Stabilization}

Let $q(t)=-\nabla f(y(t)),\;
\mathcal M=\gamma\Id+\partial\phi.$
The operator $\mathcal M$ is maximal and strongly monotone. The equation can
be written as
$$
\dot v(t)+{\mathcal M(v(t))}\ni q(t).
$$
{The strong monotonicity of $\mathcal M$ allows us to identify the
limiting equilibrium without finite-dimensional compactness and, when
$\beta=0$, without an estimate on the acceleration.}

\begin{theorem}
\label{thm:terminal-balance}
Under Assumption~\ref{ass:basic}, every solution of
\eqref{eq:implicit-dynamics} has finite length and there exists
$x_\infty\in\Hc$ such that
\begin{equation}
x(t)\to x_\infty,
\qquad
v(t)\to0,
\qquad
y(t)\to x_\infty.
\end{equation}
Moreover,
\begin{equation}
\label{eq:cont-term-bal}
-\nabla f(x_\infty)\in\partial\phi(0),
\end{equation}
and
$$
\norm{x(t)-x_\infty}
\leq\int_t^{+\infty}\norm{v(s)}\,ds.
$$
\end{theorem}

\begin{proof}
Theorem~\ref{thm:continuous-energy} gives $v\in L^1(0,+\infty;\Hc)$.
Hence $x(t)$ is a Cauchy curve and converges strongly to a point $x_\infty$.
The displayed tail estimate follows by integrating $\dot x=v$.

Assume first that $\beta>0$. Then $\dot v\in L^2(0,+\infty;\Hc)$ and
$$
\norm{v(t)-v(s)}
\leq |t-s|^{1/2}\norm{\dot v}_{L^2(0,+\infty;\Hc)}.
$$
Thus $v$ is uniformly continuous. Its $L^1$ integrability implies
$v(t)\to0$, and therefore $y(t)\to x_\infty$.

It remains to handle $\beta=0$ without using $\dot v\in L^2$. In this case
$y=x$, so
$$
q(t)=-\nabla f(x(t))\longrightarrow
p:=-\nabla f(x_\infty).
$$
The operator $\mathcal M=\gamma\Id+\partial\phi$ is maximal
monotone by the full-domain sum theorem, and it is $\gamma$-strongly
monotone. Since $0\in\partial\phi(0)$, every
$m=\gamma v+\xi\in{\mathcal M(v)}$ satisfies
\[
\langle m,v\rangle
=\gamma\norm{v}^{2}+\langle\xi,v\rangle
\geq\gamma\norm{v}^{2}.
\]
Thus $\mathcal M$ is coercive and hence surjective. Moreover, strong
monotonicity gives
\[
\norm{{\mathcal M^{-1}(p_1)}-{\mathcal M^{-1}(p_2)}}
\leq\frac1\gamma\norm{p_1-p_2},
\]
so its inverse is single-valued and everywhere defined; see
\cite[Chapter~II]{Barbu2010}. Set
$w={\mathcal M^{-1}(p)}$. Strong monotonicity, applied to
$q(t)-\dot v(t)\in{\mathcal M(v(t))}$ and $p\in{\mathcal M(w)}$, yields
$$
\frac12\frac{d}{dt}\norm{v(t)-w}^2
+\gamma\norm{v(t)-w}^2
\leq\ip{q(t)-p}{v(t)-w}.
$$
To justify division at the zeros of $v-w$, set
$\rho_\delta(t)=(\norm{v(t)-w}^{2}+\delta^{2})^{1/2}$. The preceding
inequality implies, for almost every $t$,
\[
\dot\rho_\delta(t)+\gamma
\frac{\norm{v(t)-w}^{2}}{\rho_\delta(t)}
\leq\norm{q(t)-p}\,
\frac{\norm{v(t)-w}}{\rho_\delta(t)}.
\]
Integration on compact intervals and passage to the limit
$\delta\downarrow0$ therefore give, in the distributional and hence
almost-everywhere sense,
$$
\frac{d}{dt}\norm{v(t)-w}
+\gamma\norm{v(t)-w}
\leq\norm{q(t)-p}.
$$
Consequently, for $t\geq t_0$,
\[
\norm{v(t)-w}\leq e^{-\gamma(t-t_0)}\norm{v(t_0)-w}
+\int_{t_0}^{t}e^{-\gamma(t-s)}\norm{q(s)-p}\,ds.
\]
The integral tends to zero: split it at a time after which
$\norm{q(s)-p}$ is uniformly small, and use exponential decay on the
remaining compact part. Thus $v(t)\to w$. Since
$v\in L^1(0,+\infty;\Hc)$, necessarily $w=0$. Thus $v(t)\to0$ and
$p\in{\mathcal M(0)}=\partial\phi(0)$.

For $\beta>0$, the convergence already proved gives
$q(t)\to p$. Repeating the same strong-monotonicity argument shows that
$v(t)\to{\mathcal M^{-1}(p)}$. Since $v(t)\to0$, one again obtains
$p\in\partial\phi(0)$. This proves \eqref{eq:cont-term-bal}
in both cases.
\end{proof}

The condition used for finite-time stabilization is an inclusion of a
ball in $\partial\phi(0)$. The distance to the boundary can be used only
after membership in this set has been established.

\begin{proposition}
\label{prop:exact-sticking}
Let $T\geq0$. The following assertions are equivalent:
\begin{enumerate}[label=\textup{(\roman*)},leftmargin=2.5em]
\item the solution is constant on $[T,+\infty)$;
\item $v(T)=0$ and $-\nabla f(x(T))\in\Cfr$.
\end{enumerate}
\end{proposition}

\begin{proof}
Necessity follows by inserting the constant trajectory into the inclusion
and using the closedness of $\Cfr$. Conversely, the constant pair
$(x(T),0)$ solves the Cauchy problem starting at time $T$ whenever the
displayed balance holds. Uniqueness in Theorem~\ref{thm:well-posedness}
makes it coincide with the original trajectory.
\end{proof}

We use the following uniform interior condition on the driving force.
\begin{assumption}
\label{ass:eventual-margin}
There exist $T_0\ge0$ and $\varepsilon>0$ such that
$$
q(t)+\varepsilon\B\subseteq\Cfr,
\qquad\text{for every }t\ge T_0.
$$
\end{assumption}

This condition gives a scalar decay inequality for the speed and an
upper bound for the stopping time.

\begin{theorem}
\label{thm:finite-time-stabilization}
Under Assumption~\ref{ass:eventual-margin}, the trajectory becomes
stationary in finite time. More precisely, for almost every $t\geq T_0$ such that
$v(t)\ne0$,
\begin{equation}
\frac{d}{dt}\norm{v(t)}+\gamma\norm{v(t)}+\varepsilon\le0.
\label{eq:continuous-speed-decay}
\end{equation}
\par\smallskip\noindent
Consequently,
\begin{equation}
\Tstop
\le T_0+\frac{1}{\gamma}
\log\left(1+\frac{\gamma\norm{v(T_0)}}{\varepsilon}\right).
\label{eq:continuous-stopping-bound}
\end{equation}
\par\smallskip\noindent
Furthermore, finite-time stabilization follows from the strict
interior condition
\begin{equation}
\label{eq:strict-terminal-condition}
-\nabla f(x_\infty)\in\interior(\partial\phi(0)).
\end{equation}
\end{theorem}

\begin{proof}
Fix a time such that $v(t)\ne0$, set $u=v(t)/\norm{v(t)}$, and choose the
subgradient selection $\xi(t)$. Since
$q(t)+\varepsilon u\in\partial\phi(0)$, monotonicity gives
$$
\ip{\xi(t)-q(t)}{u}\geq\varepsilon.
$$
Pairing $\dot v+\gamma v+\xi=q$ with $u$ proves
\eqref{eq:continuous-speed-decay}. Multiplication by
$e^{\gamma(t-T_0)}$ shows that the speed must vanish no later than the
right-hand side of \eqref{eq:continuous-stopping-bound}. At its first zero,
$q(t)\in\Cfr$; Proposition~\ref{prop:exact-sticking} makes the zero
permanent.

Finally, let $p_\infty=-\nabla f(x_\infty)$ and assume
\eqref{eq:strict-terminal-condition}. Choose $\varepsilon>0$ such that
$$
p_\infty+2\varepsilon\B\subseteq\Cfr.
$$
Theorem~\ref{thm:terminal-balance} gives $q(t)\to p_\infty$, hence
$q(t)+\varepsilon\B\subseteq\Cfr$ from some time $T_0$ onward. The first
part applies.
\end{proof}

The estimate depends on the time from which the interior inclusion
holds. This time is not controlled by Assumption~\ref{ass:basic}.

\begin{remark}
The strict interior condition is sufficient but not necessary. The
bound \eqref{eq:continuous-stopping-bound} applies only from the time $T_0$
at which the uniform interior inclusion begins to hold.
\end{remark}

The strict inclusion cannot in general be replaced by a non-strict
one.
\begin{proposition}
\label{prop:boundary-no-stopping}
Let $\Hc=\R$, $\beta=0$, $\gamma=2$,
$$
\phi(v)=r|v|,
\qquad
f(x)=\frac12x^2,
$$
where $r>0$. For every $A>0$, the solution starting from
$x_0=-r-A$ and $v_0=A$ is
\begin{equation}
\label{eq:sol-exp-bord}
x(t)=-r-Ae^{-t},
\qquad
v(t)=Ae^{-t}.
\end{equation}
It converges to the boundary equilibrium $x_\infty=-r$, but it never stops
in finite time.
\end{proposition}

\begin{proof}
Since $v(t)>0$, the friction selection is constantly equal to $r$. Direct
substitution gives
$$
\dot v(t)+2v(t)+r+f'(x(t))=0.
$$
Thus \eqref{eq:sol-exp-bord} is the unique solution.
Its limiting force satisfies
$$
-f'(x_\infty)=r\in\bd([-r,r])
=\bd(\partial\phi(0)),
$$
while $v(t)>0$ for every finite $t$. In contrast, the initial pair
$(-r,0)$ is already a permanent boundary equilibrium. Hence a limiting
force on the boundary may correspond either to a trajectory that never
stops in finite time or to an equilibrium.
\end{proof}

% ============================================================================
\section{{Comparison of Implicit and Explicit Hessian-Driven Damping}}
\label{sec:finite-horizon-comparison}

For $J\in\{I,E\}$, let $(x_\beta^J,v_\beta^J)$ denote the solution of the
{implicit or explicit Hessian-driven model} with parameter $\beta$, and let $(x^0,v^0)$
be their common solution for $\beta=0$. All three trajectories start from
the same pair $(x_0,v_0)$.

The comparison requires second-order regularity and a uniform control of the
negative part of the Hessian.

\begin{assumption}
\label{ass:comparison}
The function $f$ belongs to $C^2(\Hc)$. The mapping
$\nabla^2f:\Hc\to\mathcal L(\Hc)$ is bounded on bounded sets and is
Lipschitz continuous, in the operator norm, on every bounded set. Moreover,
there exists $\kappa\geq0$
such that
$$
\ip{\nabla^2f(x)u}{u}\geq-\kappa\norm{u}^2
\qquad (x,u\in\Hc).
$$
\end{assumption}

\subsection{Finite-Horizon Estimates}

The comparison with the trajectory for $\beta=0$ is of order
$O(\beta)$. The difference between the implicit and explicit trajectories
is of order $O(\beta^2)$. The first estimate also gives the uniform bounds
used near the stopping time.

\begin{proposition}
\label{prop:uniform-small-shift}
Assume Assumptions~\ref{ass:basic} and~\ref{ass:comparison}, fix
$\bar\beta>0$ such that $\bar\beta\kappa<\gamma$, and let $T>0$. Then the
{implicit and explicit Hessian-driven models} have unique global strong solutions for
$0\leq\beta\leq\bar\beta$. Moreover, there are
constants $V_T,C_T>0$, independent of $0\leq\beta\leq\bar\beta$, such that
\begin{equation}
\label{eq:uniform-velocity-bound}
\max_{J\in\{I,E\}}\sup_{0\leq t\leq T}\norm{v_\beta^J(t)}
+\sup_{0\leq t\leq T}\norm{v^0(t)}\leq V_T
\end{equation}
and
\begin{equation}
\label{eq:suivi-ref-ord1}
\max_{J\in\{I,E\}}\sup_{0\leq t\leq T}
\left(
\norm{x_\beta^J(t)-x^0(t)}
+\norm{v_\beta^J(t)-v^0(t)}
\right)
\leq C_T\beta.
\end{equation}
\end{proposition}

\begin{proof}
For the implicit Hessian-driven model, the energy identity
\eqref{eq:continuous-dissipation} bounds $v_\beta^I$ uniformly because the
initial shifted energies are bounded for $0\leq\beta\leq\bar\beta$. For the
explicit model, write its phase-space formulation as
\[
\dot z+{\mathcal A(z)}+\mathcal B_\beta^E(z)\ni0,
\qquad
\mathcal B_\beta^E(x,v)
=\bigl(-v,\gamma v+\nabla f(x)+\beta\nabla^2f(x)v\bigr).
\]
Assumption~\ref{ass:comparison} implies that
$\mathcal B_\beta^E$ is Lipschitz continuous on every bounded subset of
$\Hc\times\Hc$: on such a set this follows from the boundedness and the
operator-norm Lipschitz continuity of $\nabla^2f$, together with the boundedness
of $v$. Truncation by the metric projection and the perturbation theorem
stated before Theorem~\ref{thm:well-posedness} therefore give a unique
maximal strong solution. Along this solution, its mechanical energy
$$
\E_0^E(t)=f(x_\beta^E(t))-\inf f+\frac12\norm{v_\beta^E(t)}^2
$$
satisfies
$$
\dot\E_0^E
+(\gamma-\beta\kappa)\norm{v_\beta^E}^2
+r\norm{v_\beta^E}\leq0.
$$
This gives a uniform velocity bound when $\bar\beta\kappa<\gamma$.
On any finite interval $[0,T]$ contained in its maximal
existence interval, integration of $\dot x_\beta^E=v_\beta^E$ yields a
uniform bound for $x_\beta^E$. Hence the phase point remains in a bounded set
as a finite maximal time is approached. The same local theorem, restarted
from a time sufficiently close to that endpoint with a truncation radius
larger than this bound, extends the solution beyond it, a contradiction.
The explicit solution is therefore global. If $f$ is convex, the term
$\beta\ip{\nabla^2f(x_\beta^E)v_\beta^E}{v_\beta^E}$ is nonnegative and no
coupling condition is needed. The reference estimate is the case $\beta=0$.
Integrating $\dot x=v$ then bounds all positions on $[0,T]$.

Let $L_T$ be a common Lipschitz constant of $\nabla f$ on a ball containing
these trajectories and all their shifted states.

We prove \eqref{eq:suivi-ref-ord1} for IH; the EH proof is
identical. Set $X=x_\beta^I-x^0$ and $V=v_\beta^I-v^0$, and choose the
subgradient selections in the two equations. Monotonicity of
$\partial\phi$ and
$$
\norm{\nabla f(x_\beta^I+\beta v_\beta^I)-\nabla f(x^0)}
\leq L_T\norm{X}+\beta L_T V_T
$$
give, after pairing the velocity equation with $V$ and $\dot X=V$ with
$X$,
$$
\frac{d}{dt}\bigl(\norm{X}^2+\norm{V}^2\bigr)
\leq c_T\bigl(\norm{X}^2+\norm{V}^2\bigr)+c_T\beta^2.
$$
The initial difference is zero, so Gronwall's inequality proves the IH
estimate. For EH, use instead
$$
\norm{\nabla f(x_\beta^E)-\nabla f(x^0)
+\beta\nabla^2f(x_\beta^E)v_\beta^E}
\leq L_T\norm{X}+\beta L_T V_T.
$$
\end{proof}

Taylor's formula gives the difference between the shifted-gradient
force and the explicit Hessian-driven force. Define
$$
R_\beta(x,v)
=\nabla f(x+\beta v)-\nabla f(x)-\beta\nabla^2f(x)v.
$$
Under Assumption~\ref{ass:comparison},
\begin{equation}
\nabla f(x+\beta v)
=\nabla f(x)+\beta\nabla^2f(x)v+R_\beta(x,v),
\end{equation}
where
$$
R_\beta(x,v)
=\beta\int_0^1
\bigl(\nabla^2f(x+s\beta v)-\nabla^2f(x)\bigr)v\,ds.
$$
In particular, continuity of the Hessian gives
$\norm{R_\beta(x(t),v(t))}=o(\norm{v(t)})$ along every bounded trajectory
for which $v(t)\to0$. The Lipschitz continuity of the Hessian gives the
uniform second-order estimate needed below.

\begin{lemma}
\label{lem:taylor-defect}
Under Assumption~\ref{ass:comparison}, let $B\subset\Hc$ be a bounded convex
set on which the Hessian has Lipschitz constant $M_B$. If
$x,x+\beta v\in B$, then
\begin{equation}
\label{eq:taylor-defect}
\norm{R_\beta(x,v)}
\leq \frac{M_B}{2}\beta^2\norm{v}^2.
\end{equation}
\end{lemma}

\begin{proof}
Estimate the integrand in the preceding representation by
$M_Bs\beta\norm{v}^2$ and integrate over $[0,1]$.
\end{proof}

This estimate gives the following comparison of the trajectories on a
fixed time interval.

\begin{theorem}
\label{thm:finite-horizon-shadowing}
Under the assumptions of Proposition~\ref{prop:uniform-small-shift}, for
every $T>0$ there exists a constant $C_T>0$, independent of
$\beta\in[0,\bar\beta]$, such that
\begin{equation}
\label{eq:finite-horizon-shadowing}
\sup_{0\leq t\leq T}
\left(
\norm{x_\beta^I(t)-x_\beta^E(t)}
+\norm{v_\beta^I(t)-v_\beta^E(t)}
\right)
\leq C_T\beta^2.
\end{equation}
\end{theorem}

\begin{proof}
Set $X=x_\beta^I-x_\beta^E$ and $V=v_\beta^I-v_\beta^E$. Rewrite the IH
force as
$$
\nabla f(x_\beta^I)+\beta\nabla^2f(x_\beta^I)v_\beta^I
+R_\beta(x_\beta^I,v_\beta^I).
$$
On $[0,T]$, Proposition~\ref{prop:uniform-small-shift} and
Lemma~\ref{lem:taylor-defect} give
$$
\norm{R_\beta(x_\beta^I,v_\beta^I)}
\leq \frac{M_T}{2}V_T^2\beta^2,
$$
where $M_T$ is a Hessian Lipschitz constant on a ball containing all relevant
states. If $L_T$ denotes a bound for the Hessian on the same ball, then
the remaining smooth-force difference satisfies
\begin{align*}
&\norm{\nabla f(x_\beta^I)-\nabla f(x_\beta^E)
+\beta\nabla^2f(x_\beta^I)v_\beta^I
-\beta\nabla^2f(x_\beta^E)v_\beta^E}\\
&\hspace{4em}\leq
(L_T+\bar\beta M_T V_T)\norm{X}
+\bar\beta L_T\norm{V}.
\end{align*}
Subtracting the equations, using monotonicity of $\partial\phi$, and pairing
with $(X,V)$ therefore yields
$$
\frac{d}{dt}\bigl(\norm{X}^2+\norm{V}^2\bigr)
\leq K_T\bigl(\norm{X}^2+\norm{V}^2\bigr)+C_T\beta^4.
$$
Since $X(0)=V(0)=0$, Gronwall's inequality proves
\eqref{eq:finite-horizon-shadowing}.
\end{proof}
The constant obtained in the proof depends on $T$ through a Gronwall
estimate. Thus Theorem~\ref{thm:finite-horizon-shadowing} gives a comparison
on bounded time intervals. Its extension to $[0,+\infty[$ requires the
additional stopping condition introduced below.
% ============================================================================
\subsection{{Comparison of Permanent Stopping Times}}
\label{sec:stopping-shadowing}

Set $\Cfr=\partial\phi(0)$ and define the total driving forces
\begin{equation}
\label{eq:two-driving-forces}
q_\beta^I(t)=-\nabla f(x_\beta^I(t)+\beta v_\beta^I(t)),
\qquad
q_\beta^E(t)=-\nabla f(x_\beta^E(t))
-\beta\nabla^2f(x_\beta^E(t))v_\beta^E(t).
\end{equation}

We assume that the trajectory for $\beta=0$ stops permanently and that
its equilibrium force satisfies a strict interior condition.

\begin{assumption}
\label{ass:strict-reference-stop}
The reference trajectory $(x^0,v^0)$ has a finite permanent stopping
time $T^0>0$. There exists $\delta>0$ such that
\begin{equation}
\label{eq:strict-reference-margin}
-\nabla f(x^0(T^0))+\delta\B\subseteq \Cfr.
\end{equation}
\end{assumption}

The following lemma compares two stopping times. It does not use
Taylor's formula and will also be used for the time discretization.

\begin{lemma}
\label{lem:event-stability}
Let $(x_j,v_j)$, $j=1,2$, solve
$$
\dot v_j+\gamma v_j+\partial\phi(v_j)\ni q_j(t),
\qquad \dot x_j=v_j,
$$
with measurable forces $q_j\in L^1(a,b;\Hc)$, and assume that their
permanent stopping times $T_j$ belong to $[a,b]$. Assume that
for some $\varepsilon>0$,
$$
q_j(t)+\varepsilon\B\subseteq \Cfr
\qquad (t\in[a,b],\ j=1,2).
$$
If $\displaystyle\sup_{a\leq t\leq b}\norm{v_1(t)-v_2(t)}\leq\eta,$
then
\begin{equation}
\label{eq:borne-arret-abs}
|T_1-T_2|\leq\frac{\eta}{\varepsilon}.
\end{equation}
\end{lemma}

\begin{proof}
Suppose $T_1<T_2$. Since $v_1(T_1)=0$, one has
$\norm{v_2(T_1)}\leq\eta$. The second velocity cannot vanish on
$[T_1,T_2)$: indeed, if $v_2(t_0)=0$ for some
$t_0\in[T_1,T_2)$, then the tube assumption gives $q_2(t)\in\Cfr$ for
every $t\in[t_0,b]$. Hence the function $\widetilde v\equiv0$ solves the
nonautonomous velocity inclusion on $[t_0,b]$. This solution is unique:
if $v$ and $\widetilde v$ have the same value at $t_0$, subtraction of
their inclusions with the same forcing $q_2$ and monotonicity of
$\partial\phi$ give
\[
\frac12\frac{d}{dt}\norm{v-\widetilde v}^{2}
+\gamma\norm{v-\widetilde v}^{2}\leq0
\quad\text{a.e. on }(t_0,b),
\]
and Gronwall's inequality gives $v=\widetilde v$. Thus $t_0$ would be a
permanent stopping time, contradicting $t_0<T_2$. For almost every
$t\in[T_1,T_2)$, monotonicity of
$\partial\phi$ gives
$$
\frac{d}{dt}\norm{v_2(t)}+\gamma\norm{v_2(t)}+\varepsilon\leq0.
$$
Integration backward from $v_2(T_2)=0$ yields
$$
\norm{v_2(T_1)}
\geq\frac{\varepsilon}{\gamma}
\bigl(e^{\gamma(T_2-T_1)}-1\bigr)
\geq\varepsilon(T_2-T_1).
$$
This proves \eqref{eq:borne-arret-abs}; the other ordering is
identical.
\end{proof}

The strict interior condition implies that the two perturbed stopping
times remain close to $T^0$.
\begin{proposition}
\label{prop:stopping-localization}
Suppose that Assumption~\ref{ass:strict-reference-stop} and the hypotheses of
Proposition~\ref{prop:uniform-small-shift} hold. Then there exist
$\tau,\varepsilon,\beta_0,C>0$ such that, for
$J\in\{I,E\}$ and $0<\beta\leq\beta_0$,
\begin{equation}
\label{eq:persistent-event-margin}
q_\beta^J(t)+\varepsilon\B\subseteq \Cfr
\qquad\text{for }t\in[T^0-\tau,T^0+\tau],
\end{equation}
the permanent stopping time $T_\beta^J$ is finite, and
\begin{equation}
\label{eq:loc-arret-ord1}
|T_\beta^J-T^0|\leq C\beta.
\end{equation}
\end{proposition}

\begin{proof}
Write $p^0(t)=-\nabla f(x^0(t))$. By continuity and
\eqref{eq:strict-reference-margin}, one can choose $\tau>0$ and
$\varepsilon>0$ such that
$$
p^0(t)+3\varepsilon\B\subseteq \Cfr
\qquad (T^0-\tau\leq t\leq T^0+\tau).
$$
The reference trajectory is constant after $T^0$. {We first prove a
force estimate on the whole interval used below.} On
the fixed interval $[0,T^0+\tau]$, Proposition
\ref{prop:uniform-small-shift} and the uniform velocity bound give
$$
\norm{q_\beta^I-p^0}
\leq L\bigl(\norm{x_\beta^I-x^0}
+\beta\norm{v_\beta^I}\bigr),\qquad
\norm{q_\beta^E-p^0}
\leq L\norm{x_\beta^E-x^0}
+\beta H\norm{v_\beta^E},
$$
where $L$ and $H$ are, respectively, a Lipschitz constant of $\nabla f$ and
a bound for $\nabla^2f$ on a common bounded ball containing all the relevant
states. Consequently,
\begin{equation}
\label{eq:suivi-force-ord1}
\max_{J\in\{I,E\}}
\sup_{0\leq t\leq T^0+\tau}
\norm{q_\beta^J(t)-p^0(t)}\leq C\beta.
\end{equation}
After decreasing $\beta_0$, this estimate and the preceding interior
inclusion prove \eqref{eq:persistent-event-margin}.

We now show that neither perturbed trajectory can stop far before
$T^0$. Set
$\mathcal S=\{0\}\times\Cfr$. For every $t<T^0$, $
\bigl(v^0(t),p^0(t)\bigr)\notin\mathcal S.$
Otherwise $v^0(t)=0$ and $p^0(t)\in\Cfr$, so the exact sticking criterion
and uniqueness would make the reference trajectory constant from $t$,
contradicting the definition of $T^0$. The map
$t\mapsto(v^0(t),p^0(t))$ is continuous and $\mathcal S$ is closed. Hence
$$
d_\tau:=\min_{0\leq t\leq T^0-\tau}
\dist\bigl((v^0(t),p^0(t)),\mathcal S\bigr)>0.
$$
Combining \eqref{eq:suivi-ref-ord1} with
\eqref{eq:suivi-force-ord1}, and using the product norm on
$\Hc\times\Hc$, yields
$$
\sup_{0\leq t\leq T^0-\tau}
\norm{(v_\beta^J(t),q_\beta^J(t))-(v^0(t),p^0(t))}
\leq C\beta.
$$
For $C\beta<d_\tau/2$, the distance of
$(v_\beta^J(t),q_\beta^J(t))$ to $\mathcal S$ is at least $d_\tau/2$ on
$[0,T^0-\tau]$. At any permanent stopping time of either perturbed model,
the corresponding pair $(v_\beta^J,q_\beta^J)$ belongs to $\mathcal S$.
Therefore neither perturbed trajectory can stop on this earlier interval.

At time $T^0$, the first-order state estimate gives
$\norm{v_\beta^J(T^0)}\leq C\beta$. If the perturbed trajectory has not yet
stopped, \eqref{eq:persistent-event-margin} and the scalar extinction
inequality imply
$$
T_\beta^J-T^0
\leq\frac1\gamma
\log\left(1+\frac{\gamma C\beta}{\varepsilon}\right)
\leq C_1\beta.
$$
For small $\beta$, the right-hand side is less than $\tau$, so the
preceding argument remains inside the interval on which the common interior
inclusion holds.
At the resulting zero, the corresponding force belongs to $\Cfr$.
The constant continuation therefore solves either perturbed model, and
uniqueness makes this zero permanent.

If instead $T_\beta^J<T^0$, the exclusion just proved gives
$T_\beta^J>T^0-\tau$. On this interval the reference velocity is nonzero
before $T^0$, and the interior inclusion for the reference force gives
$$
\norm{v^0(t)}
\geq\frac{3\varepsilon}{\gamma}
\bigl(e^{\gamma(T^0-t)}-1\bigr)
\geq3\varepsilon(T^0-t).
$$
At $t=T_\beta^J$, the left-hand side is at most $C\beta$ because
$v_\beta^J(T_\beta^J)=0$. This proves the other half of
\eqref{eq:loc-arret-ord1}.

\end{proof}
We now combine the finite-time comparison with the preceding lemma.
\begin{theorem}
\label{thm:stopping-shadowing}
Under Assumptions~\ref{ass:basic}, \ref{ass:comparison}, and
\ref{ass:strict-reference-stop}, fix $\bar\beta>0$ such that
$\bar\beta\kappa<\gamma$. There exist
$\beta_0,C>0$ such that, for $0<\beta\leq\beta_0$,
\begin{equation}
\label{eq:temps-arret-ord2}
|T_\beta^I-T_\beta^E|\leq C\beta^2,
\end{equation}
\begin{equation}
\label{eq:global-shadowing}
\sup_{t\geq0}
\left(
\norm{x_\beta^I(t)-x_\beta^E(t)}
+\norm{v_\beta^I(t)-v_\beta^E(t)}
\right)
\leq C\beta^2,
\end{equation}
and
\begin{equation}
\label{eq:stopping-shadowing}
\norm{x_\beta^I(T_\beta^I)
-x_\beta^E(T_\beta^E)}\leq C\beta^2.
\end{equation}
Moreover, each individual stopping time satisfies
$|T_\beta^J-T^0|\leq C\beta$.
\end{theorem}

\begin{proof}
Proposition~\ref{prop:stopping-localization} places both stopping times in a
fixed interval on which the same interior inclusion holds for the two
forces, with some $\varepsilon>0$. The finite-horizon estimate
\eqref{eq:finite-horizon-shadowing} gives
$$
\sup_t\norm{v_\beta^I(t)-v_\beta^E(t)}\leq C\beta^2
$$
on that interval. Lemma~\ref{lem:event-stability} proves
\eqref{eq:temps-arret-ord2}.

Choose $T_+>T^0$ so that both trajectories have stopped before $T_+$ for
all sufficiently small $\beta$. On $[0,T_+]$, Theorem
\ref{thm:finite-horizon-shadowing} gives the right-hand side of
\eqref{eq:global-shadowing}. For $t\geq T_+$, both trajectories are
constant, and their distance equals their distance at $T_+$. This proves the
global estimate. Finally,
$$
x_\beta^I(T_\beta^I)=x_\beta^I(T_+),
\qquad
x_\beta^E(T_\beta^E)=x_\beta^E(T_+),
$$
so \eqref{eq:stopping-shadowing} follows from the same estimate.
\end{proof}

\subsection{Quadratic Exactness and Sharpness}
\label{sec:quadratic-coincidence}

For $x\mapsto
f(x)=\frac12\ip{Qx}{x}-\ip{b}{x},
$
one has
$$
\nabla f(x+\beta v)=\nabla f(x)+\beta Qv.
$$
Thus the two continuous models coincide exactly on quadratic objectives.
With $v_k=(x_k-x_{k-1})/h$, their discrete forces also coincide because
$$
h\nabla f(x_k+\beta v_k)
=h\nabla f(x_k)
+\beta\bigl(\nabla f(x_k)-\nabla f(x_{k-1})\bigr).
$$

The order $O(\beta^2)$ in Theorem~\ref{thm:stopping-shadowing} is
optimal under its assumptions.
\begin{proposition}
\label{prop:quadratic-order-optimal}
There exist a scalar strongly convex $C^\infty$ objective with globally
Lipschitz gradient and Hessian, {the simplified Coulomb potential
$\phi(v)=r|v|$}, and initial data
satisfying Assumption~\ref{ass:strict-reference-stop}, for which
\begin{equation}
\label{eq:state-optimality}
\lim_{\beta\downarrow0}
\frac{v_\beta^I(t_*)-v_\beta^E(t_*)}{\beta^2}=c_*>0
\end{equation}
at some fixed pre-stopping time $t_*>0$, and
\begin{equation}
\label{eq:stopping-time-optimality}
\lim_{\beta\downarrow0}
\frac{T_\beta^I-T_\beta^E}{\beta^2}=c_{\rm stop}>0.
\end{equation}
\end{proposition}

The proof is given in Appendix~\ref{app:sharpness-proof}.

% ============================================================================
\section{{Proximal Implicit Hessian-Driven Algorithm and}
{Finite Convergence of the Discrete Iterates}}
\label{sec:discrete-analysis}

The continuous analysis only requires local smoothness on bounded
sets. To obtain a step-size condition independent of the iterates, we now
assume that the gradient is globally Lipschitz continuous.
\begin{assumption}
\label{ass:discrete-smoothness}
The gradient of $f$ is globally $L$-Lipschitz continuous for some $L\geq0$.
\end{assumption}
\subsection{From the Implicit Hessian-Driven Dynamics to a Proximal
Algorithm}

Let $h>0$ and $t_k=kh$. We approximate $x(t_k)$ and $\dot x(t_k)$ by
$x_k$ and $v_k$, respectively. In \eqref{eq:implicit-dynamics}, the
acceleration, the viscous term, and the dry-friction term are discretized
implicitly at time $t_{k+1}$, whereas the gradient
$\nabla f(x_k+\beta v_k)$ is evaluated at time $t_k$.
This gives
\begin{equation}
\left\{
\begin{aligned}
&\frac{v_{k+1}-v_k}{h}+\gamma v_{k+1}
+\partial\phi(v_{k+1})
+\nabla f(x_k+\beta v_k)\ni0,\\
&\frac{x_{k+1}-x_k}{h}=v_{k+1},
\end{aligned}
\right.
\label{eq:discrete-inclusion}
\end{equation}
for $k\geq0$, with $x_0$ and $v_0$ prescribed. Introduce the ghost point
$x_{-1}=x_0-hv_0$. Then
$$
v_k=\frac{x_k-x_{k-1}}{h}
\qquad (k\geq0),
$$
and eliminating the velocities from \eqref{eq:discrete-inclusion} yields
the equivalent second-order discretization
\begin{equation}
\frac{x_{k+1}-2x_k+x_{k-1}}{h^2}
+\frac{\gamma}{h}(x_{k+1}-x_k)
+\partial\phi\left(\frac{x_{k+1}-x_k}{h}\right)
+\nabla f\left(x_k+\frac{\beta}{h}(x_k-x_{k-1})\right)\ni0.
\label{eq:position-discretization}
\end{equation}
for every $k\geq0$.
This formula is implicit only with respect to the dry-friction term.
The gradient at the shifted point is evaluated explicitly.

To solve \eqref{eq:discrete-inclusion}, set
$$
y_k=x_k+\beta v_k,
\qquad
g_k=\nabla f(y_k),
\qquad
a_h=\frac{1}{1+h\gamma},
\qquad
\lambda_h=\frac{h}{1+h\gamma}.
$$
Recall that, for $\lambda>0$, the proximal mapping of $\phi$ is
$$
\prox_{\lambda\phi}(p)
=\argmin_{w\in\Hc}
\left\{\lambda\phi(w)+\frac12\norm{w-p}^2\right\}
=(\Id+\lambda\partial\phi)^{-1}(p).
$$
The velocity inclusion is equivalent to
$$
0\in v_{k+1}-\bigl(a_hv_k-\lambda_hg_k\bigr)
+\lambda_h\partial\phi(v_{k+1}).
$$
By the resolvent identity, its unique solution is
\begin{equation}
p_{k+1}=a_hv_k-\lambda_hg_k,
\qquad
v_{k+1}=\prox_{\lambda_h\phi}(p_{k+1}).
\label{eq:proximal-update}
\end{equation}
We have therefore obtained the following implementable algorithm.

\begin{center}
\fbox{\begin{minipage}{0.92\linewidth}
\textbf{{\textsc{IH-IPGDF}: Implicit Hessian-Driven Inertial
Proximal-Gradient Algorithm with Dry Friction}}

\medskip
\textit{Initialization.} Choose $x_0\in\Hc$, $v_0\in\Hc$, a step size
$h>0$, and parameters $\gamma>0$, $\beta\geq0$. Set
$$
a_h=\frac{1}{1+h\gamma},
\qquad
\lambda_h=\frac{h}{1+h\gamma}.
$$
\textit{Iteration.} For $k=0,1,2,\ldots$, compute
\begin{align*}
y_k&=x_k+\beta v_k,\\
p_{k+1}&=a_hv_k-\lambda_h\nabla f(y_k),\\
v_{k+1}&=\prox_{\lambda_h\phi}(p_{k+1}),\\
x_{k+1}&=x_k+hv_{k+1}.
\end{align*}
\end{minipage}}
\end{center}

{For the simplified Coulomb friction law adopted here,
$\phi(v)=r\norm{v}$,} the proximal step is the vector
soft-thresholding operation
$$
v_{k+1}
=\left(1-\frac{\lambda_hr}{\norm{p_{k+1}}}\right)_{+}p_{k+1},
$$
with $v_{k+1}=0$ when $p_{k+1}=0$.

The three discrete methods used in the numerical comparison fit the same
proximal recursion. Set $x_0^J=x_0$ and $v_0^J=v_0$ for
$J\in\{D,E,I\}$, and let $x_{-1}^E=x_0-hv_0$. For $k\geq0$, define
$$
G_k^{D}=\nabla f(x_k^D),\qquad
G_k^{E}=\nabla f(x_k^E)
+\frac{\beta}{h}\bigl(\nabla f(x_k^E)-\nabla f(x_{k-1}^E)\bigr),\qquad
G_k^{I}=\nabla f(x_k^I+\beta v_k^I).
$$
The DF, EH, and IH updates are, respectively,
\begin{equation}
p_{k+1}^J=a_hv_k^J-\lambda_hG_k^J,
\qquad
v_{k+1}^J=\prox_{\lambda_h\phi}(p_{k+1}^J),
\qquad J\in\{D,E,I\},
\label{eq:three-discrete-updates}
\end{equation}
followed by $x_{k+1}^J=x_k^J+hv_{k+1}^J$. The EH initialization therefore
uses the single stored value $\nabla f(x_{-1})$.

Each iteration uses one gradient evaluation and one proximal
evaluation. Any additional evaluation of $\nabla f(x_k)$ used in the
numerical plots is not part of the algorithm.

The semi-implicit discretization of inertia, viscous damping, and dry
friction follows the construction of \cite{AdlyAttouch2021}. The
algorithm studied here differs in the smooth force: it uses one gradient at
$x_k+\beta v_k$, whereas the explicit Hessian scheme of
\cite{AdlyAttouch2020} uses a difference of consecutive gradients. The two
updates coincide on quadratic objectives, as shown in
Subsection~\ref{sec:quadratic-coincidence}.

\subsection{{Discrete Energy and Step-Size Conditions}}

Set
$$
d_{k+1}=v_{k+1}-v_k
$$
and
\begin{equation}
\E_k^h
=f(y_k)-\inf_{\Hc}f
+\frac{1+\beta\gamma}{2}\norm{v_k}^2
+\beta\phi(v_k).
\label{eq:discrete-energy}
\end{equation}
When $\beta=0$, the last term is omitted rather than interpreted as
$0\cdot(+\infty)$.
The identity
$$
\frac{y_{k+1}-y_k}{h}
=v_{k+1}+\frac{\beta}{h}d_{k+1}
$$
determines the test direction in the descent calculation.

\begin{proposition}
\label{prop:discrete-energy}
Suppose that $v_k\in\dom(\phi)$. Then every update
\eqref{eq:discrete-inclusion} satisfies
\begin{equation}
\E_{k+1}^h-\E_k^h+h\phi(v_{k+1})
+Q_h(v_{k+1},d_{k+1})\le0,
\label{eq:discrete-dissipation}
\end{equation}
\par\smallskip\noindent
where
$$
Q_h(v,d)
=A_h\norm{v}^2+C_h\norm{d}^2-Lh\beta\ip{v}{d},\;
A_h=h\gamma-\frac{Lh^2}{2},\;
C_h=\frac12+\frac{\beta}{h}
+\frac{\beta\gamma}{2}-\frac{L\beta^2}{2}.
$$
\end{proposition}

\begin{proof}
Choose $\xi_{k+1}\in\partial\phi(v_{k+1})$ so that
\begin{equation}
\label{eq:discrete-selected-equation}
d_{k+1}+h\gamma v_{k+1}+h\xi_{k+1}+hg_k=0.
\end{equation}
Pair this equality with
$$
v_{k+1}+\frac\beta h d_{k+1}
=\frac{y_{k+1}-y_k}{h}.
$$
The increment and viscous terms satisfy
\begin{align*}
\ip{d_{k+1}}{v_{k+1}}
&=\frac12\left(\norm{v_{k+1}}^2-\norm{v_k}^2
+\norm{d_{k+1}}^2\right),\\
\frac\beta h\norm{d_{k+1}}^2
&=\frac\beta h\norm{d_{k+1}}^2,\\
h\gamma\norm{v_{k+1}}^2
+\beta\gamma\ip{v_{k+1}}{d_{k+1}}
&=h\gamma\norm{v_{k+1}}^2\\
&\quad+\frac{\beta\gamma}{2}
\left(\norm{v_{k+1}}^2-\norm{v_k}^2
+\norm{d_{k+1}}^2\right).
\end{align*}
The two subgradient inequalities, respectively at $0$ and at $v_k$, give
$$
h\ip{\xi_{k+1}}{v_{k+1}}
\geq h\phi(v_{k+1}),\quad
\beta\ip{\xi_{k+1}}{d_{k+1}}
\geq\beta\bigl(\phi(v_{k+1})-\phi(v_k)\bigr).
$$
Finally, the $L$-smoothness of $f$ and
$y_{k+1}-y_k=hv_{k+1}+\beta d_{k+1}$ imply
\begin{equation*}
h\ip{g_k}{v_{k+1}}+\beta\ip{g_k}{d_{k+1}}
\geq f(y_{k+1})-f(y_k)
-\frac L2
\norm{hv_{k+1}+\beta d_{k+1}}^2.
\end{equation*}
The last square is exactly
$$
h^2\norm{v_{k+1}}^2+\beta^2\norm{d_{k+1}}^2
+2h\beta\ip{v_{k+1}}{d_{k+1}}.
$$
Substitution in \eqref{eq:discrete-selected-equation} and collection of the
three coefficients give
$$
A_h=h\gamma-\frac{Lh^2}{2},
\quad
C_h=\frac12+\frac\beta h+\frac{\beta\gamma}{2}
-\frac{L\beta^2}{2},
\quad
-Lh\beta,
$$
which proves \eqref{eq:discrete-dissipation}.
\end{proof}

Unlike the continuous energy relation
\eqref{eq:continuous-dissipation}, estimate
\eqref{eq:discrete-dissipation} is an inequality because the smooth term is
estimated by the descent lemma.

The domain condition is automatically satisfied after one proximal update.

\begin{remark}
The proximal optimality condition gives
$v_{k+1}\in\Dom(\partial\phi)\subseteq\dom(\phi)$ for every $k$. Hence, if
$v_0\notin\dom(\phi)$, Proposition~\ref{prop:discrete-energy} applies from
$k=1$ onward. When $\beta=0$, the term $\beta\phi(v_k)$ is absent and the
one-step calculation does not require $\phi(v_k)<+\infty$.
\end{remark}

The quadratic form is represented by
\begin{equation}
M_h=
\begin{pmatrix}
A_h & -Lh\beta/2\\
-Lh\beta/2 & C_h
\end{pmatrix}.
\label{eq:lyapunov-matrix}
\end{equation}
For $L>0$, $\gamma>0$, and $\beta\ge0$, define
\begin{equation}
h_{\mathrm{Lyap}}
=\frac{\gamma-L\beta+
\sqrt{(\gamma-L\beta)^2+
4L\beta\gamma/(1+\beta\gamma)}}{L}.
\label{eq:h-lyap}
\end{equation}

The determinant of $M_h$ gives the exact condition for this matrix to
be positive semidefinite.
\begin{theorem}
\label{thm:lyapunov-certificate}
Let $\beta\geq0$ and $h>0$.
\begin{enumerate}[label=\textup{(\roman*)},leftmargin=2.5em]
\item If $L>0$ and $\gamma>0$, then
$$
M_h\succeq0
\quad\Longleftrightarrow\quad
0<h\le h_{\mathrm{Lyap}},
$$
and
$$
M_h\succ0
\quad\Longleftrightarrow\quad
0<h<h_{\mathrm{Lyap}}.
$$
For $\beta=0$, one has $h_{\mathrm{Lyap}}=2\gamma/L$.
\item If $L=0$ and $\gamma>0$, then $M_h\succ0$ for every $h>0$.
\item If $L=\gamma=0$, then $M_h\succeq0$ but $M_h\not\succ0$ for every
$h>0$.
\item If $\gamma=0<L$, then $M_h\not\succeq0$ for every $h>0$.
\end{enumerate}
\end{theorem}

\begin{proof}
Assume first that $L>0$ and $\gamma>0$. Direct expansion, with no estimate,
gives
\begin{equation}
\label{eq:lyapunov-determinant}
\det M_h
=\frac{1+\beta\gamma}{4}P(h),
\qquad
P(h)=\frac{4\beta\gamma}{1+\beta\gamma}
+2(\gamma-L\beta)h-Lh^2.
\end{equation}
The polynomial $P$ is strictly concave. Its positive root is precisely
$h_{\mathrm{Lyap}}$ in \eqref{eq:h-lyap}. When $\beta>0$, its other root is
negative, whereas for $\beta=0$ its roots are $0$ and $2\gamma/L$.
Consequently,
$$
\det M_h\geq0
\quad\Longleftrightarrow\quad
0<h\leq h_{\mathrm{Lyap}}.
$$

It remains to check the diagonal conditions, which cannot be inferred from
the determinant alone. Since
$$
P\left(\frac{2\gamma}{L}\right)
=-\frac{4\beta^2\gamma^2}{1+\beta\gamma}\leq0,
$$
one has $h_{\mathrm{Lyap}}\leq2\gamma/L$, with strict inequality when
$\beta>0$. Thus
$$
A_h=h\left(\gamma-\frac{Lh}{2}\right)\geq0
$$
for $0<h\leq h_{\mathrm{Lyap}}$. For $0<h<h_{\mathrm{Lyap}}$, one has
$A_h>0$ and $\det M_h>0$, hence $M_h\succ0$. At
$h=h_{\mathrm{Lyap}}$ and $\beta>0$, one still has $A_h>0$ while
$\det M_h=0$, so $M_h\succeq0$ with rank one. At the remaining endpoint
$\beta=0$, $h=2\gamma/L$, the matrix is
$$
M_h=
\begin{pmatrix}
0&0\\[0.2em]
0&1/2
\end{pmatrix}.
$$
If $h>h_{\mathrm{Lyap}}$, the determinant is negative; hence the matrix is
not positive semidefinite. This proves~(i).

If $L=0$, the matrix is diagonal with
$$
A_h=h\gamma,
\qquad
C_h=\frac12+\frac\beta h+\frac{\beta\gamma}{2}>0.
$$
This proves~(ii) and~(iii). Finally, if $\gamma=0<L$, then
$A_h=-Lh^2/2<0$, so positive semidefiniteness is impossible. This
proves~(iv).
\end{proof}

The condition $h\leq h_{\mathrm{Lyap}}$ is sufficient for the
decrease of the energy. It is not asserted to be necessary for the
stability of the algorithm.

\begin{remark}
The number $h_{\mathrm{Lyap}}$ is the exact boundary for the positive
semidefiniteness of the matrix \eqref{eq:lyapunov-matrix}. It gives a
sufficient condition for the decrease of the discrete energy and for
convergence. It need not be the exact stability boundary of the nonlinear
algorithm.
\end{remark}

\subsection{{Asymptotic and Finite Convergence of the Discrete Iterates}}

Positive semidefiniteness is sufficient for the finite-length argument; strict
coercivity is not required for the basic convergence conclusion.

\begin{theorem}
\label{thm:discrete-convergence}
Assume $M_h\succeq0$ and $v_0\in\dom(\phi)$. Then
\begin{equation}
\label{eq:discrete-finite-length}
hr\sum_{k=0}^{+\infty}\norm{v_{k+1}}
\leq h\sum_{k=0}^{+\infty}\phi(v_{k+1})
\leq\E_0^h.
\end{equation}
Consequently, there exists $x_\infty\in\Hc$ such that
\begin{align}
&x_k\to x_\infty,
\qquad
v_k\to0,
\qquad
d_{k+1}\to0,
\qquad
y_k\to x_\infty
\\
\label{eq:discrete-terminal-balance}
&-\nabla f(x_\infty)\in\partial\phi(0).
\end{align}
Moreover, if $M_h\succ0$, then
$$
\sum_{k=0}^{+\infty}
\left(\norm{v_{k+1}}^2+\norm{d_{k+1}}^2\right)<+\infty.
$$
\end{theorem}

\begin{proof}
Since $Q_h\geq0$, summing \eqref{eq:discrete-dissipation} from $k=0$ to
$N$ gives
$$
\E_{N+1}^h
+h\sum_{k=0}^N\phi(v_{k+1})
\leq\E_0^h.
$$
The energy is nonnegative and $\phi(v)\geq r\norm{v}$, so monotone
convergence as $N\to+\infty$ proves \eqref{eq:discrete-finite-length}.
Hence
$$
\sum_{k=0}^{+\infty}\norm{x_{k+1}-x_k}
=h\sum_{k=0}^{+\infty}\norm{v_{k+1}}<+\infty,
$$
and $(x_k)$ converges strongly to some $x_\infty$. The summability in
\eqref{eq:discrete-finite-length} gives $v_k\to0$, whence
$d_{k+1}=v_{k+1}-v_k\to0$ and $y_k=x_k+\beta v_k\to x_\infty$.
Therefore $g_k\to\nabla f(x_\infty)$.

The selection in \eqref{eq:discrete-selected-equation} satisfies
$$
\xi_{k+1}
=-g_k-\gamma v_{k+1}-\frac1h d_{k+1}
\longrightarrow-\nabla f(x_\infty).
$$
Since $v_{k+1}\to0$ and the graph of the maximally monotone operator
$\partial\phi$ is strong to weak sequentially closed, the strong convergence displayed
above yields \eqref{eq:discrete-terminal-balance}.

If $M_h\succ0$, let $\lambda_h^{\min}>0$ denote its smallest eigenvalue.
Then
$$
Q_h(v,d)\geq\lambda_h^{\min}
\left(\norm{v}^2+\norm{d}^2\right).
$$
Keeping $Q_h$ in the summed energy inequality proves the final assertion.
\end{proof}

The same convergence argument remains valid when the initial velocity enters
the effective domain only after the first update.

\begin{remark}
If $v_0\notin\dom(\phi)$, the first proximal update satisfies
$v_1\in\Dom(\partial\phi)\subseteq\dom(\phi)$. Thus the same proof, summed from
$k=1$, gives
$$
hr\sum_{k=1}^{+\infty}\norm{v_{k+1}}\leq\E_1^h.
$$
This is enough for all the convergence conclusions. When $\beta=0$, the
energy is understood without the term $\beta\phi(v_k)$, so the calculation
can already start at $k=0$.
\end{remark}

The limits of the continuous and discrete systems satisfy the same
equilibrium condition. When $f$ is strongly convex, this condition gives
estimates that do not depend on whether the limit is reached in finite time
or only asymptotically.

\begin{proposition}
\label{prop:precision-fort-conv}
Suppose that Assumption~\ref{ass:basic} holds and that $f$ is
$\mu$-strongly convex for some $\mu>0$, namely
\begin{equation}
\label{eq:strong-convexity-terminal}
f(y)\geq f(x)+\ip{\nabla f(x)}{y-x}
+\frac{\mu}{2}\norm{y-x}^{2}
\qquad\text{for all }x,y\in\Hc.
\end{equation}
Then $f$ has a unique minimizer $x^*$. For every $\bar x\in\Hc$, one has
\begin{equation}
\label{eq:est-grad-fort-conv}
\norm{\bar x-x^*}
\leq\frac{\norm{\nabla f(\bar x)}}{\mu},
\qquad
0\leq f(\bar x)-f(x^*)
\leq\frac{\norm{\nabla f(\bar x)}^2}{2\mu}.
\end{equation}
If, in addition, $\bar x$ satisfies
\begin{equation}
\label{eq:abstract-terminal-balance}
{-\nabla f(\bar x)\in\Cfr=\partial\phi(0)}
\end{equation}
{and the set $\partial\phi(0)$ is bounded, with}
\begin{equation}
\label{eq:rayon-ext-friction}
R_\phi:=\sup\{\norm{p}:p\in\Cfr\}<+\infty,
\end{equation}
then
\begin{equation}
\label{eq:est-unif-fort-conv}
\norm{\bar x-x^*}\leq\frac{R_\phi}{\mu},
\qquad
0\leq f(\bar x)-f(x^*)\leq\frac{R_\phi^2}{2\mu}.
\end{equation}
\end{proposition}

\begin{proof}
Fix $x_{\rm ref}\in\Hc$. Strong convexity gives
\[
f(x)\geq f(x_{\rm ref})
+\ip{\nabla f(x_{\rm ref})}{x-x_{\rm ref}}
+\frac{\mu}{2}\norm{x-x_{\rm ref}}^2,
\]
and hence $f$ is coercive. Let $(x_n)$ be a minimizing sequence. It is
bounded by coercivity, so reflexivity of the Hilbert space provides a
subsequence, not relabeled, and $x^*\in\Hc$ such that
$x_n\rightharpoonup x^*$. Since a continuous convex function is weakly
lower semicontinuous,
\[
f(x^*)\leq\liminf_{n\to\infty}f(x_n)=\inf_{\Hc}f.
\]
Thus $x^*$ is a minimizer. Strong convexity makes it unique, and
differentiability gives $\nabla f(x^*)=0$.

The gradient of a differentiable $\mu$-strongly convex function is
$\mu$-strongly monotone. Hence
$$
\mu\norm{\bar x-x^*}^2
\leq\ip{\nabla f(\bar x)-\nabla f(x^*)}{\bar x-x^*}
\leq\norm{\nabla f(\bar x)}\norm{\bar x-x^*}.
$$
If $\bar x=x^*$, the distance estimate is immediate; otherwise, division by
$\norm{\bar x-x^*}$ proves the distance estimate in
\eqref{eq:est-grad-fort-conv}. Applying
\eqref{eq:strong-convexity-terminal} with $x=\bar x$ and $y=x^*$ gives
$$
f(\bar x)-f(x^*)
\leq\ip{\nabla f(\bar x)}{\bar x-x^*}
-\frac{\mu}{2}\norm{\bar x-x^*}^2
\leq\frac{\norm{\nabla f(\bar x)}^2}{2\mu}.
$$
The equilibrium condition \eqref{eq:abstract-terminal-balance} and
\eqref{eq:rayon-ext-friction} imply
$\norm{\nabla f(\bar x)}\leq R_\phi$, and therefore yield
\eqref{eq:est-unif-fort-conv}.
\end{proof}

The continuous limit $x_\infty$ in
Theorem~\ref{thm:terminal-balance} and, for every fixed step size satisfying
the hypotheses of Theorem~\ref{thm:discrete-convergence}, the discrete limit
$x_\infty^h$ satisfy the equilibrium condition
\eqref{eq:abstract-terminal-balance}. Proposition
\ref{prop:precision-fort-conv} therefore applies to both
limits. If convergence occurs in finite time or in finitely many iterations,
the corresponding stopping position is the limit itself.

\begin{remark}
{For the simplified Coulomb friction law,
$\phi(v)=r\norm{v}$,} one has
$\Cfr=r\B$ and thus $R_\phi=r$. The constants in
\eqref{eq:est-unif-fort-conv} are optimal if only
the equilibrium condition is used. Indeed, for
$\Hc=\mathbb R$, $f(x)=\mu x^2/2$, and $\bar x=r/\mu$, condition
\eqref{eq:abstract-terminal-balance} holds at the boundary of $r\B$, while
both upper bounds in \eqref{eq:est-unif-fort-conv}
are attained.
\end{remark}

For later use, set
$$
q_k=-g_k=-\nabla f(y_k).
$$

Define the interior radius
$$
\rho_k
=\sup\left(
\{0\}\cup\left\{\rho>0:
-\nabla f(y_k)+\rho\B\subseteq\partial\phi(0)\right\}
\right).
$$
When the force is in the interior, $\rho_k$ equals its distance to the
boundary. {For the simplified Coulomb friction law,}
$$
\rho_k=r-\norm{\nabla f(y_k)}
$$
whenever this quantity is positive.

The permanent stopping index is
$$
\Kstop
=\inf\left\{K\in\N:
v_k=0\text{ for every }k\ge K\right\},
$$
with $\inf\varnothing=+\infty$.

The proximal update has an exact criterion for a zero velocity to remain
zero at every subsequent index.

\begin{proposition}
\label{prop:exact-discrete-sticking}
Let $j\in\N$. The sequence is constant from index $j$ onward if and only if
$$
v_j=0,
\qquad
-\nabla f(x_j)\in\partial\phi(0).
$$
\end{proposition}

\begin{proof}
If the sequence is constant from $j$ onward, the inclusion at index $j$
gives the displayed balance. Conversely, $v_j=0$ implies $y_j=x_j$, and the
balance means that $v_{j+1}=0$ satisfies the proximal optimality condition.
The proximal update is unique, so $x_{j+1}=x_j$. Induction gives
$(x_k,v_k)=(x_j,0)$ for every $k\geq j$.
\end{proof}

A uniform interior condition gives a bound on the number of remaining
iterations.

\begin{theorem}
\label{thm:arret-disc-quant}
If $\rho_k>0$ and $v_{k+1}\ne0$, then
\begin{equation}
\label{eq:discrete-speed-decay}
(1+h\gamma)\norm{v_{k+1}}+h\rho_k
\leq\norm{v_k}.
\end{equation}
Suppose that $\rho_k\ge\varepsilon>0$ for every $k\ge k_0$. If
$\gamma>0$, then
\begin{equation}
\Kstop-k_0
\le
\left\lceil
\frac{\log(1+\gamma\norm{v_{k_0}}/\varepsilon)}
{\log(1+h\gamma)}
\right\rceil.
\label{eq:discrete-stopping-bound}
\end{equation}
If $\gamma=0$, then
$$
\Kstop-k_0
\le
\left\lceil\frac{\norm{v_{k_0}}}{h\varepsilon}\right\rceil.
$$
Finally, under the hypotheses of Theorem~\ref{thm:discrete-convergence}, one
has $\Kstop<+\infty$ whenever
$$
-\nabla f(x_\infty)\in\interior(\partial\phi(0))
.$$
\end{theorem}

\begin{proof}
Fix $k$ with $\rho_k>0$ and $v_{k+1}\ne0$, and set
$u_{k+1}=v_{k+1}/\norm{v_{k+1}}$. Because $\rho_k$ is defined as a
supremum, take first any $0<\rho<\rho_k$. Then
$$
q_k+\rho u_{k+1}\in\partial\phi(0).
$$
Monotonicity of $\partial\phi$, applied to this point and to
$\xi_{k+1}\in\partial\phi(v_{k+1})$, gives
$$
\ip{\xi_{k+1}-q_k}{u_{k+1}}\geq\rho.
$$
Pairing \eqref{eq:discrete-selected-equation} with $u_{k+1}$ and using
$g_k=-q_k$ yields
$$
(1+h\gamma)\norm{v_{k+1}}+h\rho
\leq\ip{v_k}{u_{k+1}}\leq\norm{v_k}.
$$
Letting $\rho\uparrow\rho_k$ proves
\eqref{eq:discrete-speed-decay}. If $\rho_k=+\infty$, the same argument
shows directly that a nonzero $v_{k+1}$ is impossible.

Assume now $\rho_k\geq\varepsilon$ for $k\geq k_0$ and set
$a=1+h\gamma$. As long as no zero has occurred, iteration of
\eqref{eq:discrete-speed-decay} gives, for $\gamma>0$,
\begin{equation}
\label{eq:discrete-affine-iteration}
\norm{v_{k_0+n}}
\leq a^{-n}
\left(\norm{v_{k_0}}+\frac{\varepsilon}{\gamma}\right)
-\frac{\varepsilon}{\gamma}.
\end{equation}
For
$$
n=\left\lceil
\frac{\log(1+\gamma\norm{v_{k_0}}/\varepsilon)}
{\log(1+h\gamma)}
\right\rceil,
$$
the right-hand side of \eqref{eq:discrete-affine-iteration} is nonpositive.
Thus a zero must occur no later than $k_0+n$. At that index the
interior condition implies $-\nabla f(x_j)\in\partial\phi(0)$ because $v_j=0$ and
$y_j=x_j$. Proposition~\ref{prop:exact-discrete-sticking} makes the zero
permanent and proves \eqref{eq:discrete-stopping-bound}.

When $\gamma=0$, the same iteration reads
$$
\norm{v_{k_0+n}}
\leq\norm{v_{k_0}}-nh\varepsilon
$$
as long as the velocity remains nonzero, which gives the stated bound.

For the final assertion, Theorem~\ref{thm:discrete-convergence} gives
$q_k=-\nabla f(y_k)\to-\nabla f(x_\infty)$. The strict interior
condition therefore gives $\varepsilon>0$ and $k_0$ such that
$q_k+\varepsilon\B\subseteq\partial\phi(0)$ for every $k\geq k_0$.
The quantitative estimate applies.
\end{proof}

% ============================================================================
\section{Consistency, Shadowing, and Quadratic Stability}
\label{sec:consistance-disc}

We now study three questions: the approximation of the continuous
implicit Hessian-driven system, the comparison of the implicit and explicit
discretizations, and the stability of the quadratic iteration.

\subsection{{Vanishing Step Size and Convergence of Stopping Times}}
\label{sec:vanishing-stepsize}

The preceding results concern a fixed step size. We now justify the passage
from the proximal recursion to {the implicit Hessian-driven dynamics} as
$h\downarrow0$.
For each $h>0$, let $(x_k^h,v_k^h)_{k\geq0}$ be generated by
\eqref{eq:discrete-inclusion} from the same initial data $(x_0,v_0)$, and
write $t_k=kh$. On $[t_k,t_{k+1}]$, define the continuous piecewise affine
interpolants
\begin{equation}
\label{eq:affine-interpolants}
\begin{aligned}
x_h(t)&=x_k^h+(t-t_k)v_{k+1}^h,\\
v_h(t)&=v_k^h+(t-t_k)a_{k+1}^h,
\qquad
a_{k+1}^h=\frac{v_{k+1}^h-v_k^h}{h}.
\end{aligned}
\end{equation}
Thus $(x_h(t_k),v_h(t_k))=(x_k^h,v_k^h)$ at every node.

{We first prove an error estimate on a bounded time interval. The order
$h^{1/2}$ is the general estimate for the implicit Euler approximation of a
maximally monotone evolution when the initial datum belongs to the operator
domain. We do not assume time differentiability of the dry-friction
selection. The same general order for ISIHD is discussed in
\cite{MaulenSotoFadiliAttouchOchs2026}. Here we prove it directly for the
proximal recursion and then study the convergence of the physical stopping
times and stopping positions.}

\begin{theorem}
\label{thm:vanishing-stepsize}
Suppose that Assumptions~\ref{ass:basic} and
\ref{ass:discrete-smoothness} hold, fix $\beta\geq0$, and let
$v_0\in\Dom(\partial\phi)$. Let $(x,v)$ be the strong solution of
\eqref{eq:implicit-dynamics}, and let $(x_h,v_h)$ be defined by
\eqref{eq:affine-interpolants}. Then, for every $T>0$, there exist
$C_T>0$ and $h_T>0$ such that
\begin{equation}
\label{eq:err-disc-hor-fin}
\sup_{0\leq t\leq T}
\left(
\norm{x_h(t)-x(t)}+\norm{v_h(t)-v(t)}
\right)
\leq C_T\sqrt h
\end{equation}
for every $0<h\leq h_T$.
\end{theorem}

\begin{proof}
Set $z=(x,v)$, $z_k^h=(x_k^h,v_k^h)$, and
$$
g(x,v)=\nabla f(x+\beta v).
$$
On the phase space, use the operators
$$
\mathcal A(x,v)=\{0\}\times\partial\phi(v),
\qquad
\mathcal B(x,v)=\bigl(-v,\gamma v+g(x,v)\bigr).
$$
The operator $\mathcal A$ is maximally monotone and, under
Assumption~\ref{ass:discrete-smoothness}, $\mathcal B$ is globally
Lipschitz continuous.

We first establish a uniform consistency bound. By
Theorem~\ref{thm:lyapunov-certificate}, $M_h\succeq0$ for every sufficiently
small $h$. The discrete energy estimate then gives a constant $V>0$,
independent of such $h$, for which
$$
\sup_{k\geq0}\norm{v_k^h}\leq V.
$$
Choose $\xi_0\in\partial\phi(v_0)$ and set
$\mathcal M=\gamma\Id+\partial\phi$. The first velocity update can be written
$$
v_1^h=(\Id+h\mathcal M)^{-1}(v_0-hg(x_0,v_0)).
$$
Since
$$
v_0=(\Id+h\mathcal M)^{-1}
\bigl(v_0+h(\gamma v_0+\xi_0)\bigr),
$$
the nonexpansiveness of the resolvent gives
\begin{equation}
\label{eq:borne-prem-acc-disc}
\norm{a_1^h}
\leq\norm{g(x_0,v_0)+\gamma v_0+\xi_0}.
\end{equation}

For $k\geq1$, select $\xi_{k+1}^h\in\partial\phi(v_{k+1}^h)$ and
$\xi_k^h\in\partial\phi(v_k^h)$ from the discrete equations at indices $k$
and $k-1$, respectively. Subtracting these equations, using the
monotonicity of $\partial\phi$, and testing in the direction of
$v_{k+1}^h-v_k^h$ gives
\begin{align*}
\norm{a_{k+1}^h}-\norm{a_k^h}
&\leq\norm{g(x_k^h,v_k^h)-g(x_{k-1}^h,v_{k-1}^h)}\\
&\leq Lh\left(\norm{v_k^h}+\beta\norm{a_k^h}\right).
\end{align*}
The same inequality is immediate when $v_{k+1}^h=v_k^h$. The discrete
Gronwall lemma and \eqref{eq:borne-prem-acc-disc} therefore
give, for every $T>0$,
\begin{equation}
\sup_{0\leq k h\leq T}\norm{a_{k+1}^h}\leq A_T,
\end{equation}
where $A_T$ is independent of small $h$. Consequently,
\begin{equation}
\label{eq:uniform-phase-increments}
\norm{z_{k+1}^h-z_k^h}\leq C_T h
\qquad (0\leq kh\leq T).
\end{equation}

The discrete system can now be viewed as an implicit Euler approximation
with a vanishing residual:
\begin{equation}
\label{eq:perturbed-implicit-euler}
\frac{z_{k+1}^h-z_k^h}{h}
+{\mathcal A(z_{k+1}^h)}+{\mathcal B(z_{k+1}^h)}
\ni R_{k+1}^h,
\end{equation}
where
$$
R_{k+1}^h
=\left(0,
g(x_{k+1}^h,v_{k+1}^h)-g(x_k^h,v_k^h)
\right).
$$
By \eqref{eq:uniform-phase-increments} and the Lipschitz continuity of $g$,
\begin{equation}
\label{eq:vanishing-euler-residual}
\sup_{0\leq kh\leq T}\norm{R_{k+1}^h}\leq C_T h.
\end{equation}

We now prove the required implicit Euler estimate in the present setting.
This also makes explicit why the maximal-monotone order is $h^{1/2}$.
In the estimates already obtained above, replace $T$ by $T+1$ and restrict
to $0<h\leq1$; this covers the last grid interval intersecting $[0,T]$.
Set
$$
\mathcal C=\mathcal A+\mathcal B
$$
and let $\ell\geq0$ be a global Lipschitz constant of $\mathcal B$. For
$w_i\in\Dom(\mathcal C)$ and $c_i\in{\mathcal C(w_i)}$, monotonicity of
$\mathcal A$ gives
\begin{equation}
\label{eq:quasi-monotonicity-euler}
\ip{c_1-c_2}{w_1-w_2}
\geq-\ell\norm{w_1-w_2}^2.
\end{equation}
Moreover, if $h\ell<1$, then $\Id+h\mathcal C$ is onto and its inverse is
single-valued. Indeed, for each $u$ the map
$$
w\longmapsto
J_{h\mathcal A}\bigl(u-h{\mathcal B(w)}\bigr),
\qquad
J_{h\mathcal A}=(\Id+h\mathcal A)^{-1},
$$
is a contraction with constant $h\ell$. Thus the unperturbed implicit Euler
sequence $\widehat z_0=z_0$,
\begin{equation}
\label{eq:euler-aux-nonpert}
\frac{\widehat z_{k+1}-\widehat z_k}{h}
+\widehat c_{k+1}=0,
\qquad
\widehat c_{k+1}\in{\mathcal C(\widehat z_{k+1})},
\end{equation}
is well defined for all sufficiently small $h$.

We first record uniform bounds for the continuous and discrete
phase-space velocities. Since $z_0\in\Dom(\mathcal C)$, choose
$c_0\in{\mathcal C(z_0)}$. If $c(t)=-\dot z(t)\in{\mathcal C(z(t))}$, then
\eqref{eq:quasi-monotonicity-euler}, applied to $(z(t),c(t))$ and
$(z_0,c_0)$, gives
$$
\frac{d}{dt}\norm{z(t)-z_0}
\leq\norm{c_0}+\ell\norm{z(t)-z_0}
$$
for almost every $t$ at which $z(t)\ne z_0$; the general case follows by
regularizing the norm. Hence
$$
\norm{z(s)-z_0}\leq s e^{\ell s}\norm{c_0}.
$$
The stability estimate obtained from
\eqref{eq:quasi-monotonicity-euler} for two solutions, applied to the time
translates $t\mapsto z(t+s)$ and $t\mapsto z(t)$, then yields
$$
\norm{z(t+s)-z(t)}
\leq e^{\ell t}\norm{z(s)-z_0}
\leq s e^{\ell(t+s)}\norm{c_0}.
$$
Consequently, $z$ is Lipschitz continuous on bounded intervals and
$$
\norm{c(t)}=\norm{\dot z(t)}\leq e^{\ell T}\norm{c_0}
\quad\text{for almost every }t\in[0,T].
$$
On the discrete side, comparison of
$(\widehat z_{k+1},\widehat c_{k+1})$ and
$(\widehat z_k,\widehat c_k)$ in
\eqref{eq:quasi-monotonicity-euler}, together with
$\widehat z_{k+1}-\widehat z_k=-h\widehat c_{k+1}$, gives
$$
(1-h\ell)\norm{\widehat c_{k+1}}
\leq\norm{\widehat c_k}.
$$
The same argument at the first step uses $(z_0,c_0)$. Therefore, if
$h\ell\leq1/2$, there exists $M_T>0$, independent of $h$, such that
\begin{equation}
\label{eq:borne-phase-cont-disc}
\norm{c(t)}\leq M_T,
\qquad
\norm{\widehat c_{k+1}}\leq M_T
\end{equation}
whenever $0\leq t\leq T+1$ and $0\leq kh\leq T+1$.

We next compare the unperturbed Euler sequence with the exact trajectory at
the grid points. Put $e_k=\widehat z_k-z(t_k)$. Since
$$
e_{k+1}-e_k
=-\int_{t_k}^{t_{k+1}}
\bigl(\widehat c_{k+1}-c(t)\bigr)\,dt,
$$
we have
$$
\frac12\left(
\norm{e_{k+1}}^2-\norm{e_k}^2
+\norm{e_{k+1}-e_k}^2
\right)
=-\int_{t_k}^{t_{k+1}}
\ip{\widehat c_{k+1}-c(t)}{e_{k+1}}\,dt.
$$
For $t\in[t_k,t_{k+1}]$, write
$$
e_{k+1}
=\widehat z_{k+1}-z(t)+z(t)-z(t_{k+1}).
$$
Using 
\eqref{eq:quasi-monotonicity-euler} and
\eqref{eq:borne-phase-cont-disc}, we obtain
\begin{align*}
-\ip{\widehat c_{k+1}-c(t)}{e_{k+1}}
&\leq
\ell\norm{\widehat z_{k+1}-z(t)}^2
+\norm{\widehat c_{k+1}-c(t)}
 \norm{z(t)-z(t_{k+1})}\\
&\leq
2\ell\norm{e_{k+1}}^2
+2\ell M_T^2h^2+2M_T^2h.
\end{align*}
After integration and omission of the nonnegative increment term, this gives
$$
(1-4\ell h)\norm{e_{k+1}}^2
\leq
\norm{e_k}^2+4M_T^2(1+\ell h)h^2.
$$
Since $e_0=0$, the discrete Gronwall lemma yields, for
$4\ell h\leq1/2$,
\begin{equation}
\label{eq:unperturbed-euler-sqrt-h}
\max_{0\leq kh\leq T+h}\norm{\widehat z_k-z(t_k)}
\leq C_T\sqrt h.
\end{equation}

It remains to control the residual in
\eqref{eq:perturbed-implicit-euler}. Select
$c_{k+1}^h\in{\mathcal C(z_{k+1}^h)}$ so that
$$
\frac{z_{k+1}^h-z_k^h}{h}+c_{k+1}^h=R_{k+1}^h
$$
and set $d_k=z_k^h-\widehat z_k$. Subtracting
\eqref{eq:euler-aux-nonpert}, pairing with $d_{k+1}$, and using
\eqref{eq:quasi-monotonicity-euler} give
$$
(1-h\ell)\norm{d_{k+1}}
\leq\norm{d_k}+h\norm{R_{k+1}^h}.
$$
Because $d_0=0$, iteration of this inequality shows that
\begin{equation}
\label{eq:stab-res-euler-pert}
\max_{0\leq kh\leq T+h}\norm{z_k^h-\widehat z_k}
\leq C_T\sum_{0\leq kh\leq T}h\norm{R_{k+1}^h}.
\end{equation}
Let $\widehat z_h$ denote the piecewise affine interpolant of
$(\widehat z_k)$. Estimate
\eqref{eq:borne-phase-cont-disc} shows that both
$\widehat z_h$ and $z$ move by at most $M_Th$ inside a grid interval.
Combining 
\eqref{eq:unperturbed-euler-sqrt-h} and
\eqref{eq:stab-res-euler-pert}, and observing that the
difference between the affine interpolants is bounded by the maximum nodal
difference, we obtain
\begin{equation}
\label{eq:borne-euler-pert}
\sup_{0\leq t\leq T}\norm{z_h(t)-z(t)}
\leq C_T\left(
\sqrt h+\int_0^{T+h}\norm{R_h(s)}\,ds
\right),
\end{equation}
where $R_h=R_{k+1}^h$ on $(t_k,t_{k+1}]$. Finally,
\eqref{eq:vanishing-euler-residual} makes the integral in
\eqref{eq:borne-euler-pert} of order $h$. This proves
\eqref{eq:err-disc-hor-fin}.

\end{proof}

We next show that the approximation remains valid when the trajectory
becomes permanently stationary. For each $h>0$, define
the permanent stopping index
$$
K_{\mathrm{stop}}^h
=\inf\left\{K\in\N:\ v_k^h=0\text{ for every }k\geq K\right\}.
$$

\begin{theorem}
Under the hypotheses of Theorem~\ref{thm:vanishing-stepsize}, suppose that
the continuous trajectory has a finite permanent stopping time
$T_{\mathrm{stop}}>0$ and satisfies the strict interior condition
\begin{equation}
\label{eq:cond-stricte-h}
-\nabla f(x(T_{\mathrm{stop}}))+\delta\B
\subseteq\partial\phi(0)
\end{equation}
for some $\delta>0$. Then $K_{\mathrm{stop}}^h<+\infty$ for every
sufficiently small $h$. Moreover, there exist $C>0$ and $h_0>0$ such that
\begin{equation}
\left|hK_{\mathrm{stop}}^h-T_{\mathrm{stop}}\right|
+\norm{x_{K_{\mathrm{stop}}^h}^h-x(T_{\mathrm{stop}})}
\leq C\sqrt h
\end{equation}
for every $0<h\leq h_0$. In particular,
$$
hK_{\mathrm{stop}}^h\longrightarrow T_{\mathrm{stop}},
\qquad
x_{K_{\mathrm{stop}}^h}^h\longrightarrow x(T_{\mathrm{stop}}).
$$
The approximation estimate also holds on $[0,+\infty[$:
\begin{equation}
\label{eq:suivi-global-cont-disc}
\sup_{t\geq0}
\left(
\norm{x_h(t)-x(t)}+\norm{v_h(t)-v(t)}
\right)
\leq C\sqrt h.
\end{equation}
\end{theorem}

\begin{proof}
For brevity, set $T_*=T_{\mathrm{stop}}$ and
$$
p(t)=-\nabla f(x(t)+\beta v(t)),
\qquad
q_k^h=-\nabla f(x_k^h+\beta v_k^h).
$$
Since the continuous trajectory is constant after $T_*$, condition
\eqref{eq:cond-stricte-h} and continuity provide
$0<\tau<T_*$ and $\varepsilon>0$ such that
\begin{equation}
\label{eq:tube-arret-cont-h}
p(t)+3\varepsilon\B\subseteq\partial\phi(0)
\qquad (T_*-\tau\leq t\leq T_*+\tau).
\end{equation}
Theorem~\ref{thm:vanishing-stepsize} and the Lipschitz continuity of
$\nabla f$ give
\begin{equation}
\label{eq:approx-force-disc}
\sup_{0\leq kh\leq T_*+\tau}
\left(
\norm{v_k^h-v(kh)}+\norm{q_k^h-p(kh)}
\right)
\leq C\sqrt h.
\end{equation}

We first show that the discrete sequence cannot stop far before
$T_*$. Set
$\mathcal S=\{0\}\times\partial\phi(0)$. For $t<T_*$,
$(v(t),p(t))\notin\mathcal S$; otherwise the exact sticking criterion would
make the continuous trajectory constant from $t$. Hence
$$
d_\tau
=\min_{0\leq t\leq T_*-\tau}
\dist\bigl((v(t),p(t)),\mathcal S\bigr)>0.
$$
For small $h$, estimate \eqref{eq:approx-force-disc} shows that
$(v_k^h,q_k^h)$ stays at distance at least $d_\tau/2$ from $\mathcal S$
whenever $kh\leq T_*-\tau$. At a permanent discrete stop, however,
Proposition~\ref{prop:exact-discrete-sticking} gives
$(v_{K_{\mathrm{stop}}^h}^h,q_{K_{\mathrm{stop}}^h}^h)\in\mathcal S$.
Thus no permanent discrete stop can occur before $T_*-\tau$.

We next prove that a stop does occur shortly after $T_*$. Let
$k_0=\lceil T_*/h\rceil$. Then $v(k_0h)=0$, and
\eqref{eq:approx-force-disc} gives
\begin{equation}
\norm{v_{k_0}^h}\leq C\sqrt h.
\end{equation}
For all nodes in $[T_*-\tau,T_*+\tau]$, after decreasing $h_0$,
\eqref{eq:tube-arret-cont-h} and
\eqref{eq:approx-force-disc} imply
$$
q_k^h+2\varepsilon\B\subseteq\partial\phi(0).
$$
Define
$$
n_h=\left\lceil
\frac{\log(1+\gamma\norm{v_{k_0}^h}/\varepsilon)}
{\log(1+h\gamma)}
\right\rceil.
$$
For small $h$, the elementary bounds
$\log(1+h\gamma)\geq h\gamma/2$ and
$\log(1+s)\leq s$ give
\begin{equation}
hn_h\leq h+\frac{2}{\varepsilon}\norm{v_{k_0}^h}
\leq C\sqrt h.
\end{equation}
Thus all indices from $k_0$ to $k_0+n_h$ remain in the interval on
which the uniform interior inclusion holds. Iterating the discrete speed inequality
\eqref{eq:discrete-speed-decay} over these indices, exactly as in the proof
of Theorem~\ref{thm:arret-disc-quant}, forces a zero no later
than $k_0+n_h$. At the first zero, the force is in the interior of
$\partial\phi(0)$, so Proposition~\ref{prop:exact-discrete-sticking} makes
the zero permanent. This proves $K_{\mathrm{stop}}^h<+\infty$ and
$$
hK_{\mathrm{stop}}^h-T_*\leq C\sqrt h
$$
whenever the left-hand side is nonnegative.

Suppose now that $hK_{\mathrm{stop}}^h<T_*$. The previous exclusion places
this time in $[T_*-\tau,T_*)$. On this interval, the continuous velocity
cannot vanish before $T_*$, and the scalar extinction inequality integrated
backward from $v(T_*)=0$ gives
$$
\norm{v(t)}
\geq\frac{3\varepsilon}{\gamma}
\left(e^{\gamma(T_*-t)}-1\right)
\geq3\varepsilon(T_*-t).
$$
At $t=hK_{\mathrm{stop}}^h$, the discrete velocity is zero, while
\eqref{eq:approx-force-disc} gives
$\norm{v(t)}\leq C\sqrt h$. Therefore
$$
T_*-hK_{\mathrm{stop}}^h\leq C\sqrt h.
$$
This completes the stopping-time estimate.

Since $x$ is Lipschitz continuous on bounded intervals, Theorem
\ref{thm:vanishing-stepsize} and the stopping-time estimate imply
$$
\norm{x_{K_{\mathrm{stop}}^h}^h-x(T_*)}
\leq C\sqrt h.
$$
Finally, both trajectories are constant after their respective
stopping times. The estimate on $[0,T_*+\tau]$, followed by the estimate of
the stopping positions for later times, proves
\eqref{eq:suivi-global-cont-disc}.
\end{proof}

\subsection{{Comparison of the Explicit and Implicit
Hessian-Driven Discretizations}}
\label{sec:discrete-shadowing}

For a nonlinear function, the gradient at the shifted point and the
backward difference of gradients are not identical. Their difference is of
order $O(\beta^2+\beta h)$ on bounded sets.

\begin{theorem}
\label{thm:discrete-eh-ih-shadowing}
Suppose that Assumptions~\ref{ass:basic}, \ref{ass:comparison}, and
\ref{ass:discrete-smoothness} hold. Fix $T>0$, $\bar\beta>0$, and
$\bar h>0$. Let $(x_k^E,v_k^E)$ and $(x_k^I,v_k^I)$ be generated by the EH
and IH recursions in \eqref{eq:three-discrete-updates}, with the same initial
data and with $x_{-1}^E=x_0-hv_0$. Then there exists $C_T>0$, independent of
$\beta$ and $h$, such that
\begin{equation}
\label{eq:discrete-eh-ih-shadowing}
\max_{0\leq kh\leq T}
\left(
\norm{x_k^I-x_k^E}+\norm{v_k^I-v_k^E}
\right)
\leq C_T\left(\beta^2+\beta h\right)
\end{equation}
for every $0\leq\beta\leq\bar\beta$ and
$0<h\leq\bar h$. If $f$ is quadratic, the two complete discrete
trajectories coincide for every $\beta\geq0$ and $h>0$.
\end{theorem}

\begin{proof}
We first obtain bounds that are uniform with respect to
$0\leq\beta\leq\bar\beta$ and $0<h\leq\bar h$ on the fixed physical
horizon. Since $0$ minimizes $\phi$, one has
$\prox_{\lambda\phi}(0)=0$. The nonexpansiveness of the proximal mapping,
$a_h\leq1$, and $\lambda_h\leq h$ give
\begin{equation}
\label{eq:borne-vit-suivi-disc}
\norm{v_{k+1}^J}
\leq\norm{v_k^J}+h\norm{G_k^J},
\qquad J\in\{E,I\}.
\end{equation}
For EH, the identity $x_k^E-x_{k-1}^E=hv_k^E$, including at $k=0$ by the
choice of the ghost point, gives
\begin{equation}
\label{eq:eh-integral-force}
G_k^E
=\nabla f(x_k^E)
+\beta\int_0^1
\nabla^2f(x_k^E-shv_k^E)v_k^E\,ds.
\end{equation}
Because $\nabla f$ is globally $L$-Lipschitz continuous,
$\norm{\nabla^2f(x)}\leq L$. Hence, for both $J=E$ and $J=I$,
$$
\norm{G_k^J}
\leq c_0+L\norm{x_k^J}+\bar\beta L\norm{v_k^J},
$$
where $c_0=\norm{\nabla f(0)}$. Combining this estimate with
$x_{k+1}^J=x_k^J+hv_{k+1}^J$ and
\eqref{eq:borne-vit-suivi-disc}, the discrete Gronwall
lemma yields a constant $R_T>0$ such that
\begin{equation}
\label{eq:bornes-disc-eh-ih}
\max_{J\in\{E,I\}}
\max_{0\leq kh\leq T+h}
\left(\norm{x_k^J}+\norm{v_k^J}\right)
\leq R_T.
\end{equation}

For $x,v\in\Hc$, introduce the two forces
$$
\mathcal G_\beta^I(x,v)=\nabla f(x+\beta v),\qquad
\mathcal G_{\beta,h}^E(x,v)=\nabla f(x)
+\beta\int_0^1\nabla^2f(x-shv)v\,ds.
$$
All points occurring in these integrals along the trajectories belong to a
fixed bounded ball determined by $R_T$, $\bar\beta$, and $\bar h$. Let
$M_T$ be a Lipschitz constant of $\nabla^2f$ on that ball. The integral
Taylor formula gives the exact defect representation
\begin{align}
\mathcal R_{\beta,h}(x,v)
&:=\mathcal G_\beta^I(x,v)-\mathcal G_{\beta,h}^E(x,v)\notag\\
&=\beta\int_0^1
\left(
\nabla^2f(x+s\beta v)-\nabla^2f(x-shv)
\right)v\,ds.
\end{align}
Therefore
\begin{equation}
\label{eq:borne-defaut-force}
\norm{\mathcal R_{\beta,h}(x,v)}
\leq\frac{M_T}{2}\beta(\beta+h)\norm{v}^2.
\end{equation}

On the same bounded set, the explicit force is uniformly
Lipschitz in the phase variables. Indeed, for
$(x_i,v_i)$, $i=1,2$, one obtains from the integral representation
\begin{align*}
\norm{\mathcal G_{\beta,h}^E(x_1,v_1)
-\mathcal G_{\beta,h}^E(x_2,v_2)}
&\leq L\norm{x_1-x_2}
+\beta L\norm{v_1-v_2}
+\beta M_T\norm{v_2}
\left(\norm{x_1-x_2}+h\norm{v_1-v_2}\right)\\
&\leq C_T
\left(\norm{x_1-x_2}+\norm{v_1-v_2}\right).
\end{align*}
Set
$$
X_k=x_k^I-x_k^E,
\qquad
V_k=v_k^I-v_k^E.
$$
Combining the preceding Lipschitz estimate with
\eqref{eq:borne-defaut-force} and
\eqref{eq:bornes-disc-eh-ih} gives
\begin{equation}
\label{eq:comp-force-disc}
\norm{G_k^I-G_k^E}
\leq C_T\left(
\norm{X_k}+\norm{V_k}+\beta^2+\beta h
\right).
\end{equation}

The two proximal updates use the same parameters $a_h$ and $\lambda_h$.
Their nonexpansiveness and \eqref{eq:comp-force-disc} imply
$$
\norm{V_{k+1}}
\leq(1+C_Th)\norm{V_k}
+C_Th\norm{X_k}
+C_Th(\beta^2+\beta h).
$$
Moreover,
$$
\norm{X_{k+1}}
\leq\norm{X_k}+h\norm{V_{k+1}}.
$$
After increasing $C_T$, these two inequalities yield
$$
\norm{X_{k+1}}+\norm{V_{k+1}}
\leq(1+C_Th)
\left(\norm{X_k}+\norm{V_k}\right)
+C_Th(\beta^2+\beta h).
$$
Since $X_0=V_0=0$, the discrete Gronwall lemma proves
\eqref{eq:discrete-eh-ih-shadowing}.

If $f$ is quadratic, its Hessian is constant and
$\mathcal R_{\beta,h}\equiv0$. The two forces agree whenever the phase
variables agree, and induction from the common initial data gives
$(x_k^I,v_k^I)=(x_k^E,v_k^E)$ for every $k$.
\end{proof}

\begin{remark}
The two terms in \eqref{eq:discrete-eh-ih-shadowing} have different origins.
The term $\beta^2$ is the Taylor remainder of the shifted gradient,
whereas $\beta h$ is the temporal consistency defect of the backward
gradient difference. In particular, if $h=O(\beta)$, the two
discretizations differ by $O(\beta^2)$ on every bounded physical time
interval.
\end{remark}

\subsection{{Exact Modal Stability of the Frictionless Quadratic
Iteration}}
\label{sec:quadratic-modal-analysis}

Let $Q:\Hc\to\Hc$ be bounded, self-adjoint, and positive definite, and set
$$
f(x)=\frac12\ip{Qx}{x}-\ip{b}{x},
\qquad
0<\mu\Id\preceq Q\preceq L\Id.
$$
For an exact uniform boundary, $L$ is taken to be
$\max\sigma(Q)=\norm{Q}$. If a larger Lipschitz bound is used instead, the
same formulas remain sufficient but need not be sharp.
We first consider $\phi\equiv0$. The resulting iteration is linear, so
the spectral modes of $Q$ can be studied separately. The calculation gives
the exact Schur boundary for the quadratic iteration without dry friction.
The translation $z=x-Q^{-1}b$ removes the affine term. The scalar symbol
associated with $q\in\sigma(Q)$ is
$$
T_q=
\begin{pmatrix}
1-\dfrac{h^2q}{1+h\gamma}
& \dfrac{h(1-hq\beta)}{1+h\gamma}\\[1.1em]
-\dfrac{hq}{1+h\gamma}
& \dfrac{1-hq\beta}{1+h\gamma}
\end{pmatrix}.
$$
We call the quadratic iteration uniformly Schur stable when
$$
\sup_{q\in\sigma(Q)}\rho(T_q)<1.
$$

For dry friction, let $\phi(v)=r\norm{v}$ and set
$$
w_k=(\Id-h\beta Q)v_k-hQz_k.
$$
The proximal update becomes
$$
v_{k+1}
=\frac{1}{1+h\gamma}
\left(1-\frac{hr}{\norm{w_k}}\right)_{+}w_k,
\qquad
z_{k+1}=z_k+hv_{k+1}.
$$
Here $v_{k+1}=0$ when $w_k=0$.
It is nonlinear, and its equilibrium set is
$$
\mathcal E_r=\{(z,0):\norm{Qz}\leq r\}.
$$
In dimension greater than one, the radial proximal step couples the spectral
modes. Thus the matrix $T_q$ and the boundary $h_{\mathrm{quad}}$ describe
only the iteration without dry friction. Convergence with dry friction is
given by Theorems~\ref{thm:discrete-convergence}
and~\ref{thm:arret-disc-quant}.

Although the uniform result below assumes $Q\succeq\mu\Id$ with $\mu>0$,
the scalar classification includes $q=0$ to show why a zero mode
prevents Schur stability.

\begin{theorem}
\label{thm:classif-modale}
Let $q\geq0$, $h>0$, and $\gamma,\beta\geq0$.
\begin{enumerate}[label=\textup{(\roman*)},leftmargin=2.5em]
\item If $q>0$ and $\gamma+q\beta>0$, then $T_q$ is Schur stable if and
only if
\begin{equation}
\label{eq:quadratic-stability}
qh(h+2\beta)<4+2h\gamma.
\end{equation}
At equality, $-1$ is an eigenvalue and the other eigenvalue lies strictly
inside the unit disk. If the left-hand side is larger, $T_q$ has a real
eigenvalue strictly smaller than $-1$.
\item If $q=0$, the eigenvalues are $1$ and $(1+h\gamma)^{-1}$. Hence the
mode is never Schur stable. For $\gamma>0$ it is semistable; for
$\gamma=0$, the matrix is a nontrivial Jordan block at $1$.
\item If $q>0$ and $\gamma=\beta=0$, then $T_q$ is not Schur stable for
any $h>0$. It is marginally stable for $0<h<2/\sqrt q$. At
$h=2/\sqrt q$, it has a defective double eigenvalue $-1$, and for
$h>2/\sqrt q$ its spectral radius is larger than one.
\end{enumerate}
\end{theorem}

\begin{proof}
Write
$$
\tau_q=\operatorname{tr}T_q,
\qquad
\delta_q=\det T_q.
$$
A direct calculation gives
\begin{equation}
\label{eq:modal-trace-determinant}
\tau_q=
\frac{2+h\gamma-hq(h+\beta)}{1+h\gamma},
\qquad
\delta_q=\frac{1-hq\beta}{1+h\gamma}.
\end{equation}
The characteristic polynomial is
$p_q(\lambda)=\lambda^2-\tau_q\lambda+\delta_q$. For a real monic
quadratic, the Jury criterion is
\begin{equation}
\label{eq:jury-three-conditions}
1-\tau_q+\delta_q>0,
\qquad
1+\tau_q+\delta_q>0,
\qquad
1-\delta_q>0.
\end{equation}
Here the three left-hand sides are exactly
\begin{equation}
\label{eq:jury-factorization}
\frac{h^2q}{1+h\gamma},
\qquad
\frac{4+2h\gamma-qh(h+2\beta)}{1+h\gamma},
\qquad
\frac{h(\gamma+q\beta)}{1+h\gamma}.
\end{equation}
For $q>0$ and $\gamma+q\beta>0$, the first and third quantities are
strictly positive, so the middle inequality is equivalent to
\eqref{eq:quadratic-stability}.

At equality, $p_q(-1)=1+\tau_q+\delta_q=0$. Thus one eigenvalue is $-1$.
Moreover,
$$
p_q(1)=1-\tau_q+\delta_q>0,
$$
which gives $-\delta_q<1$, while the third Jury quantity gives
$\delta_q<1$, hence $-\delta_q>-1$. Thus the second eigenvalue lies
strictly inside the unit disk.
If the left-hand side of \eqref{eq:quadratic-stability} is larger,
$p_q(-1)<0$ while $p_q(\lambda)\to+\infty$ as
$\lambda\to-\infty$; hence there is a real root below $-1$.

For $q=0$, the matrix is upper triangular and the two displayed
eigenvalues follow immediately. If also $\gamma=0$, it equals
$$
\begin{pmatrix}1&h\\0&1\end{pmatrix},
$$
which proves~(ii).

Finally, if $q>0$ and $\gamma=\beta=0$, then $\delta_q=1$. For
$0<h<2/\sqrt q$, the trace lies in $(-2,2)$, so the two distinct
eigenvalues are complex conjugates of modulus one. At $h=2/\sqrt q$, the
characteristic polynomial is $(\lambda+1)^2$ and $T_q\ne-\Id$, hence the
double eigenvalue is defective. Above that value, the argument using
$p_q(-1)<0$ again gives a root below $-1$. This proves~(iii).
\end{proof}

The spectral theorem turns the scalar classification into a uniform one.

\begin{corollary}
Assume $0<\mu=\min\sigma(Q)\leq\max\sigma(Q)=L$.
If $(\gamma,\beta)\ne(0,0)$, the frictionless quadratic iteration is
uniformly Schur stable if and only if
\begin{equation}
\label{eq:cond-quad-unif}
0<h<h_{\mathrm{quad}},
\qquad
h_{\mathrm{quad}}
=\frac{\gamma-L\beta+
\sqrt{(\gamma-L\beta)^2+4L}}{L}.
\end{equation}
At $h=h_{\mathrm{quad}}$, the $q=L$ symbol has the eigenvalue $-1$; for
$h>h_{\mathrm{quad}}$ it is unstable. If $\gamma=\beta=0$, the iteration
is marginally stable, but not Schur stable, precisely for
$0<h<2/\sqrt L$.
\end{corollary}

\begin{proof}
For fixed $h>0$ and $\beta\geq0$, the left-hand side of
\eqref{eq:quadratic-stability} is increasing in $q$. Thus the worst mode is
$q=L$. Solving
$$
Lh(h+2\beta)<4+2h\gamma
$$
for $h>0$ gives \eqref{eq:cond-quad-unif}. Strictness and the
classification at and above the boundary follow from
Theorem~\ref{thm:classif-modale}. When
$\gamma=\beta=0$, apply part~(iii) uniformly over
$q\in[\mu,L]$.
\end{proof}

We now compare the sufficient condition obtained from the Lyapunov
analysis with the exact quadratic condition.

\begin{proposition}
\label{prop:lyapunov-modal-gap}
Let $L>0$, $\gamma>0$, and $\beta\geq0$. Then the Lyapunov boundary is
strictly conservative:
\begin{equation}
\label{eq:strict-certificate-gap}
0<h_{\mathrm{Lyap}}<h_{\mathrm{quad}},
\end{equation}
and the gap is exactly
\begin{equation}
\label{eq:certificate-gap-formula}
h_{\mathrm{quad}}-h_{\mathrm{Lyap}}
=\frac{4}{
(1+\beta\gamma)\left[
\sqrt{(\gamma-L\beta)^2+4L}
+\sqrt{(\gamma-L\beta)^2+
\dfrac{4L\beta\gamma}{1+\beta\gamma}}
\right]}.
\end{equation}
If $\gamma=0<L$ and $\beta>0$, the matrix $M_h$ is not positive
semidefinite for any positive step, whereas the modal Schur interval
$0<h<h_{\mathrm{quad}}$ is nonempty.
\end{proposition}

\begin{proof}
Set $a=\gamma-L\beta$. The two boundaries have the same affine part
$a/L$, while their radicands differ because
$$
0\leq\frac{\beta\gamma}{1+\beta\gamma}<1.
$$
If $\beta=0$, then $h_{\mathrm{Lyap}}=2\gamma/L>0$. If $\beta>0$, the
second radicand in \eqref{eq:h-lyap} is strictly larger than $a^2$, which
again gives $h_{\mathrm{Lyap}}>0$. The strict comparison with
$h_{\mathrm{quad}}$ follows from the displayed inequality.
Rationalizing the difference of the two square roots gives
\eqref{eq:certificate-gap-formula}. The last assertion follows from
Theorem~\ref{thm:lyapunov-certificate}(iv) and from
\eqref{eq:cond-quad-unif}.
\end{proof}

\begin{remark}
{For the quadratic objective of this section and the simplified Coulomb
friction law $\phi(v)=r\norm{v}$,}
Proposition~\ref{prop:precision-fort-conv} gives
$$
\norm{\bar x-Q^{-1}b}\leq\frac{r}{\mu},
\qquad
f(\bar x)-f(Q^{-1}b)\leq\frac{r^2}{2\mu}
$$
for every continuous or discrete limit point $\bar x$, whether it is
reached by finite-time stabilization of the continuous dynamics, by finite
convergence of the discrete iterates, or only
asymptotically. Thus these terminal accuracy estimates do not require
finite stopping.
\end{remark}
% ============================================================================
\section{Numerical Experiments}
\label{sec:numerical-validation}

\subsection{{Numerical Setting}}

We compare the dynamics without Hessian damping (DF), with explicit
Hessian-driven damping (EH), and with implicit Hessian-driven damping (IH).
We also compare their proximal discretizations. The experiments illustrate
the exact quadratic stability boundary, the equality of EH and IH for
quadratic functions, the damping effect for nonlinear functions, and the
role of the strict interior condition in finite-time stabilization and in
finite convergence of the discrete iterates.

All computations are deterministic and were performed in MATLAB R2021b.
{The simplified Coulomb friction potential is treated by its exact
proximal map. It is never smoothed.}
For the fine-step approximations of the continuous dynamics we use the forces
$$
\nabla f(x),\qquad
\nabla f(x)+\beta\nabla^2f(x)v,
\qquad
\nabla f(x+\beta v)
$$
for DF, EH, and IH, respectively. For the algorithms, IH is precisely
\eqref{eq:proximal-update}, whereas EH replaces the Hessian-vector product by
the stored gradient difference $\beta(\nabla f(x_k)-\nabla f(x_{k-1}))$.
Thus, after initialization, each algorithm requires one new gradient and one
proximal evaluation per iteration. A trajectory is declared permanently
stopped when the proximal velocity update returns zero and the
equilibrium condition $\norm{\nabla f(x)}\leq r$ is satisfied. The state is then kept
fixed.

\subsection{{Quadratic Identity and Modal Stability}}
We first consider a quadratic function without dry friction, with $
\mu=10^{-2},\; L=1,\; \gamma=1.$
Figure~\ref{fig:stab-modale-quad} displays the result. At
$\beta=0.4$, the two thresholds are
$$
h_{\rm Lyap}=1.8259,
\qquad
h_{\rm quad}=2.6881.
$$
The iteration decays both below $h_{\rm Lyap}$ and in the nonempty interval
$(h_{\rm Lyap},h_{\rm quad})$, whereas it grows after crossing the exact
modal boundary. Thus the condition obtained from the nonlinear
Lyapunov analysis is sufficient but not necessary for quadratic stability.

\begin{figure}[t]
\centering
\includegraphics[width=0.98\linewidth]{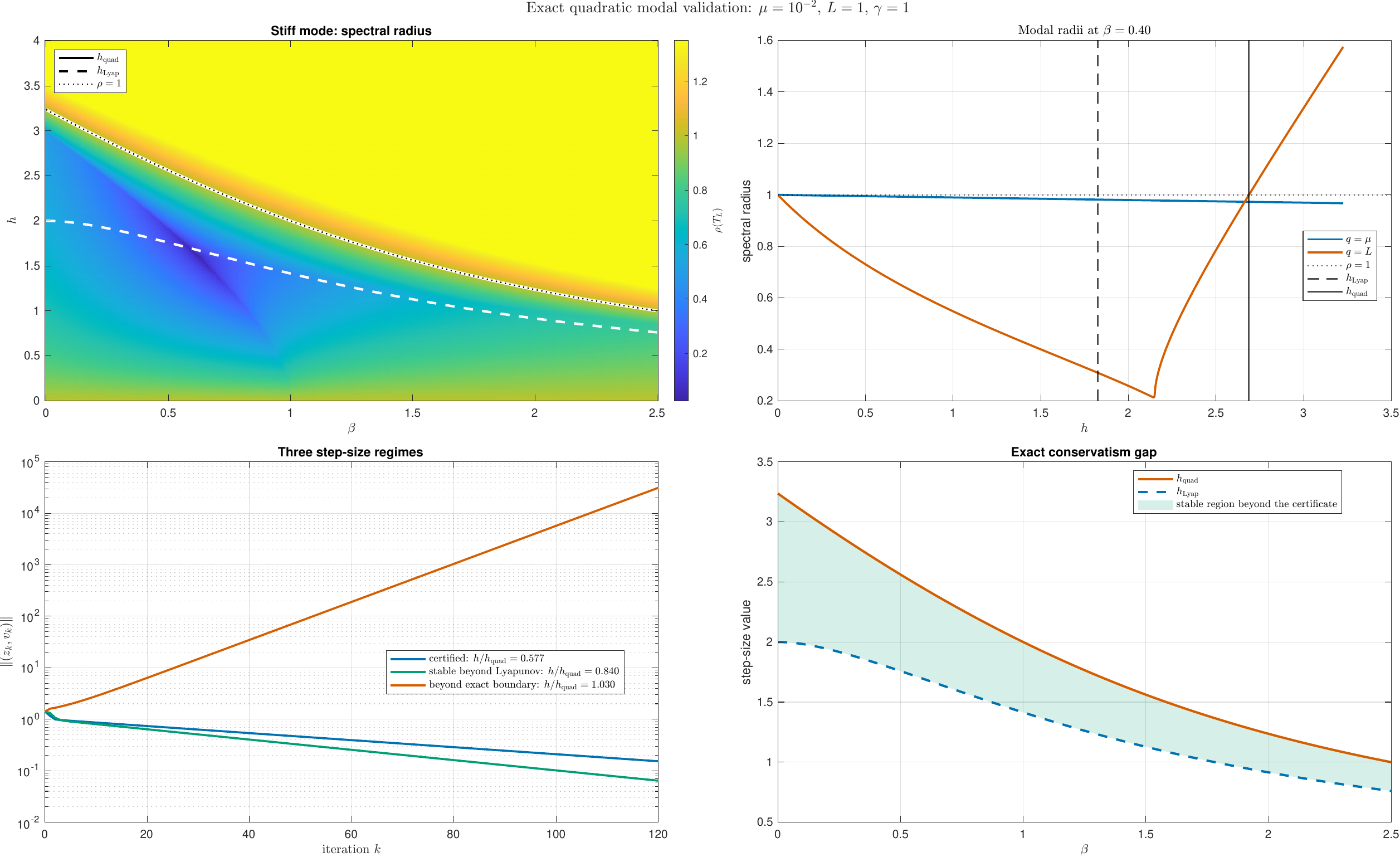}
\caption{{Quadratic modal stability.} The upper-left panel compares the
spectral-radius transition with $h_{\rm quad}$ and $h_{\rm Lyap}$. The
upper-right panel identifies the stiff mode as the limiting one. The lower
panels exhibit stable trajectories on both sides of the sufficient
Lyapunov bound
and instability beyond the exact boundary.}
\label{fig:stab-modale-quad}
\end{figure}

{We next include the simplified Coulomb friction law} and use a rotated positive-definite
quadratic with condition number $100$, with
$\gamma=1$, $\beta=0.4$, $r=0.1$, and $h=0.05$. The DF algorithm stops at
$t=7.35$, whereas EH and IH both stop at $t=4.45$. More importantly, their
complete discrete trajectories agree up to machine precision. This
illustrates the exact EH-IH identity on quadratic objectives, including the
{corresponding proximal dry-friction step} and permanent stopping.

\subsection{Nonlinear Damping Effects}

\paragraph{Curved Valley.}
We begin with a globally smooth curved valley. To compare the two
Hessian-driven terms outside the quadratic setting, we
consider
\begin{equation}
\label{eq:numerical-curved-valley}
f(x)=\frac{\mu}{2}\norm{x}^{2}+a(1-\cos x_1)
+c\bigl(1-\cos(x_2-\alpha\sin x_1)\bigr),
\end{equation}
with $(\mu,a,c,\alpha)=(1,0.2,0.6,1.2).$
This objective is nonquadratic and has a curved valley, while its gradient is
globally Lipschitz. The Hessian row sums give the valid bound $L\leq3.504$.
{We take $x_0=(-0.5,3)$, $v_0=0$, $\gamma=1.5$, $\beta=0.25$, and
$r=0.03$. The continuous dynamics are computed on $0\leq t\leq12$ with
$h=2\times10^{-3}$. The algorithms use the visibly coarser step
$h=5\times10^{-2}$ and a maximum of $250$ iterations.}

As shown in Figure~\ref{fig:curved-valley-comparison}, both Hessian corrections
reduce the overshoot near the valley floor. {The fine-step stopping times
are $7.986$ for DF, $6.020$ for EH, and $6.068$ for IH. The discrete
sequences become constant after $163$, $120$, and $121$ iterations,
respectively.} The same reduction is observed for the implicit system and
for its proximal algorithm. The small difference between EH and IH, which
is absent for quadratic functions, is now visible.

\begin{figure}[H]
\centering
\includegraphics[width=0.98\linewidth]{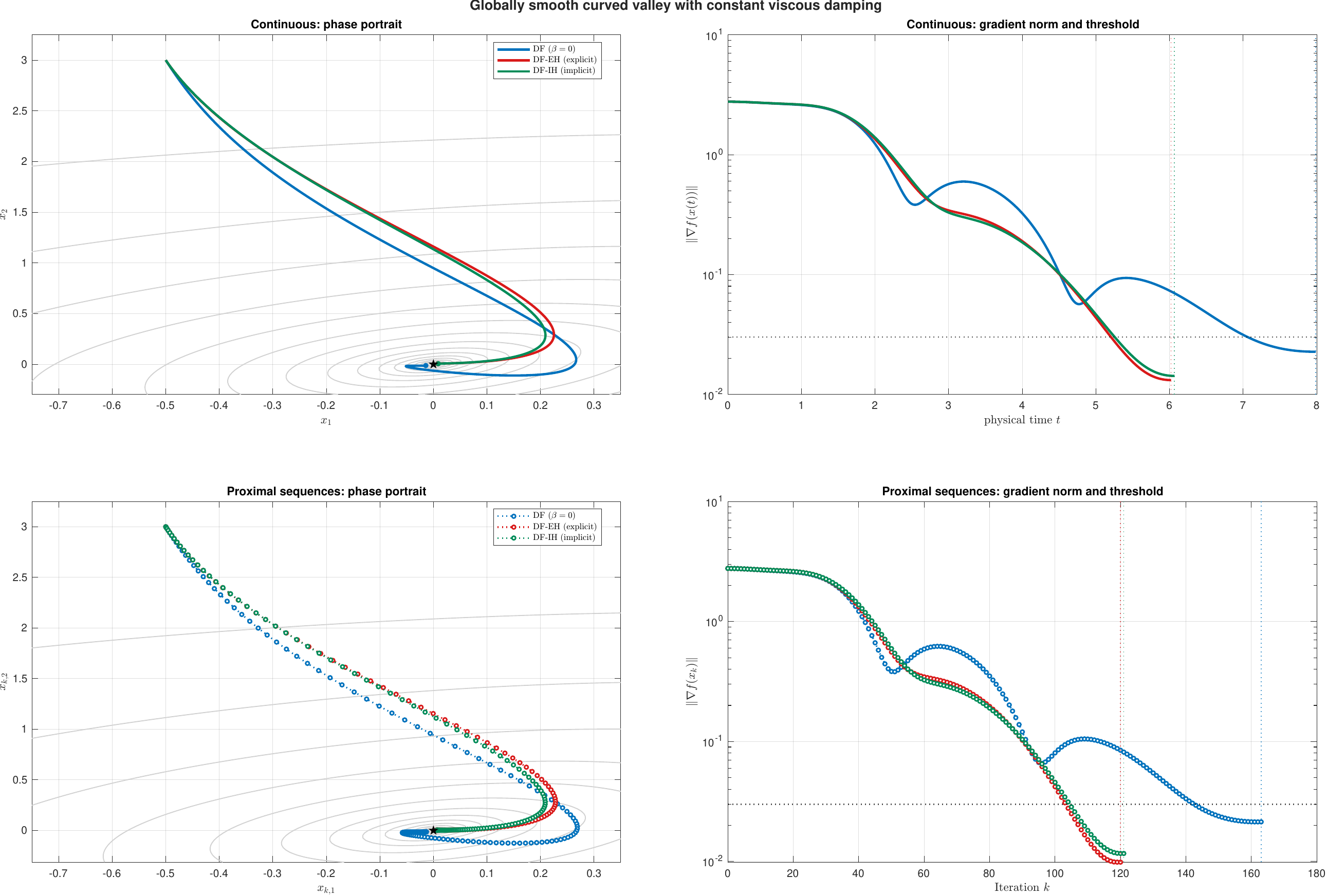}
\caption{Globally smooth curved valley. {The top row shows the continuous
trajectories as functions of the physical time $t$. The bottom row shows the
proximal sequences: open circles mark the computed iterates and dotted
segments connect consecutive iterates. Its axes are $(x_{k,1},x_{k,2})$ in
the phase portrait and the iteration index $k$ in the gradient plot.} The
{horizontal dotted line is the dry-friction threshold $r$;} vertical dotted lines
mark the detected permanent stopping times and indices. {Here $r=0.03$,
the continuous horizon is $T=12$, and the discrete budget is $250$
iterations.}}
\label{fig:curved-valley-comparison}
\end{figure}

{The remaining functions are convex, coercive, and smooth, but their
Hessians are unbounded. They satisfy the local assumptions of the continuous
analysis. The global step-size condition of
Section~\ref{sec:discrete-analysis} does not apply, so the discrete
trajectories are shown only for comparison.}

\paragraph{Anisotropic Quartic.}
Consider the strongly anisotropic quartic
\begin{equation}
\label{eq:numerical-example-one}
f(x_1,x_2)=x_1^4+10^4x_2^4.
\end{equation}
We use $x_0=(-1,-1)$, $v_0=0$, $\gamma=1$, $\beta=0.05$, and $r=0.03$.
The fine-step dynamics use $h=2\times10^{-5}$ and the algorithms use
$h=10^{-4}$. Figure~\ref{fig:example-one} makes the damping effect especially
clear. Without Hessian correction, the trajectory performs many large
transverse oscillations and stops at $t=12.237$. EH and IH suppress
these oscillations and stop at $t=8.863$ and $t=8.808$, a reduction of about
$28\%$. The corresponding algorithmic stopping times are $12.229$, $8.863$,
and $8.808$. {The stopping times obtained with the two step sizes are
close.}

\begin{figure}[t]
\centering
\includegraphics[width=0.98\linewidth]{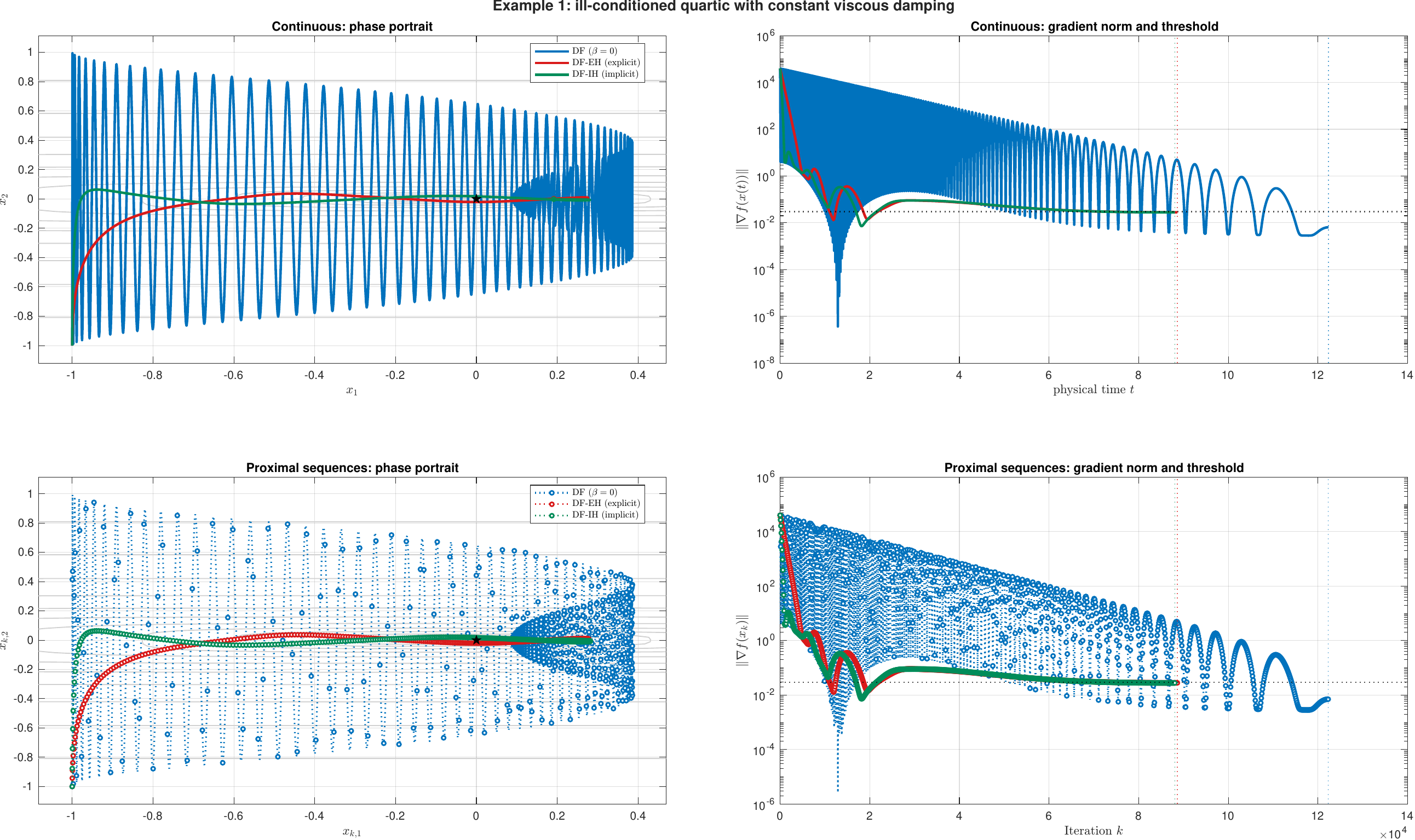}
\caption{Anisotropic quartic with fixed viscous damping. The ill-conditioned direction
produces sustained oscillations for DF, while EH and IH rapidly align the
trajectory with the flat direction. {The top row shows the continuous
trajectories as functions of $t$. The bottom row is indexed by $k$, with
phase variables $(x_{k,1},x_{k,2})$. Dotted segments connect consecutive
iterates, and open circles mark a uniformly selected subset of the computed
iterates because the discretization uses a very fine step.} The two rows use
the same physical parameters.
}
\label{fig:example-one}
\end{figure}
\FloatBarrier

\paragraph{Dependence on the Shift Parameter.}
To expose the dependence on the shift parameter, we also consider the scaled
objective $10^{-2}(x_1^4+10^4x_2^4)$ with $\gamma=1$ and $r=10^{-4}$ over the
fixed horizon $[0,45]$. Figure~\ref{fig:fixed-gamma-sweep} uses the same nine
values
$$
\beta\in\{0,0.1,0.25,0.5,1,1.75,2.75,3.75,5\}
$$
in all four panels, and identical axes throughout. The black curve is the
common $\beta=0$ baseline. Increasing $\beta$ first damps the stiff-direction
oscillations and straightens the path; for larger values it also changes the
longitudinal progress. The figure is therefore interpreted as
a finite-horizon geometric comparison, not as evidence that the stopping
time must be monotone in $\beta$. The continuous and algorithmic panels show
the same qualitative transition.

\begin{figure}[t]
\centering
\includegraphics[width=0.98\linewidth]{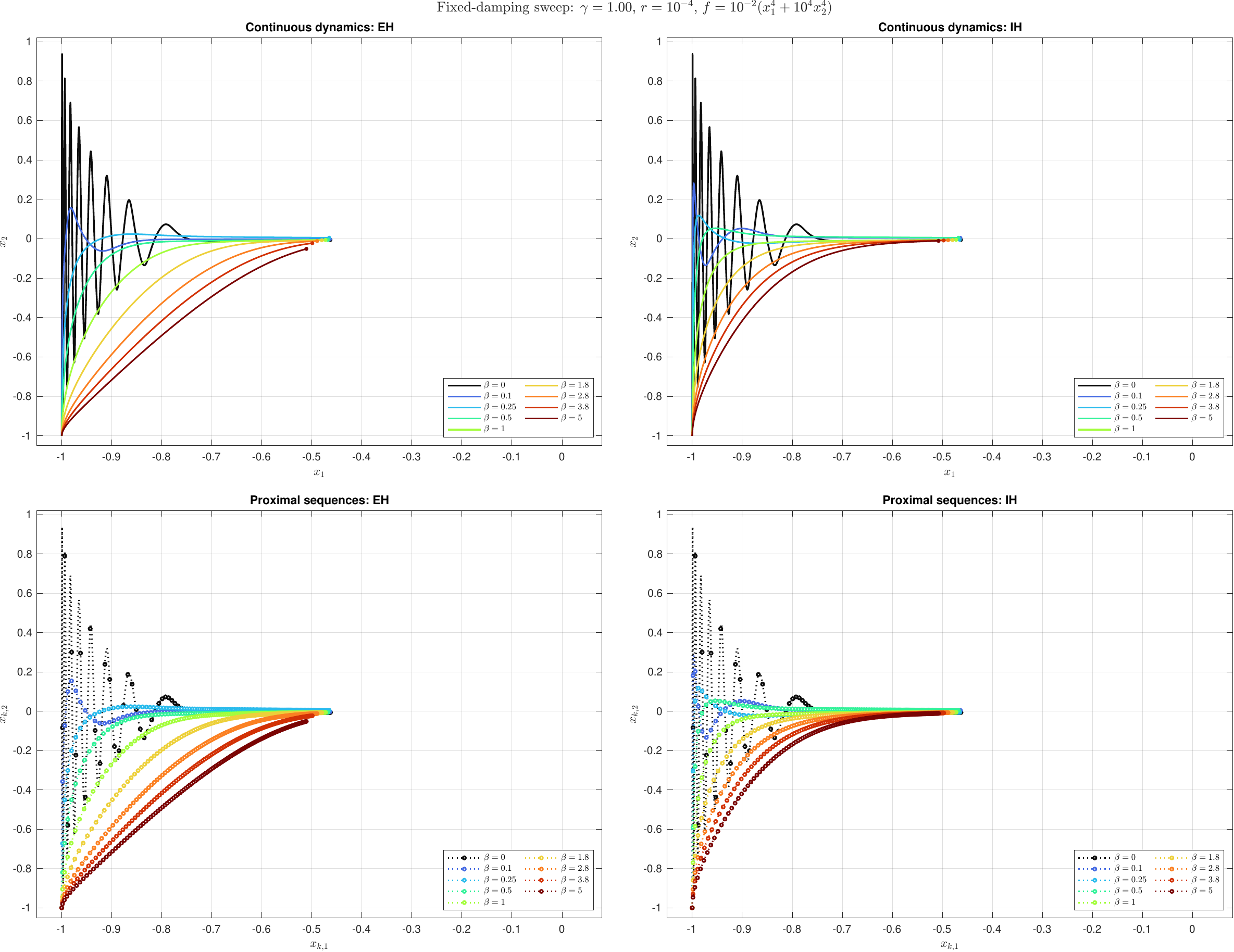}
\caption{Fixed-$\gamma$ sweep on the scaled ill-conditioned quartic. The
explicit (left) and implicit (right) corrections are compared for the
continuous dynamics (top) and proximal algorithms (bottom). The common black
curve corresponds to $\beta=0$. {Each panel contains the legend for the
tested values of $\beta$. In the bottom row, open circles mark computed
iterates, dotted segments connect consecutive iterates, and the phase
variables are $(x_{k,1},x_{k,2})$.}}
\label{fig:fixed-gamma-sweep}
\end{figure}
\FloatBarrier

\subsection{{Boundary Case}}

We finally consider
\begin{equation}
\label{eq:numerical-example-two}
f(x_1,x_2)=x_1^4+5x_2^2-4x_1-10x_2+8,
\end{equation}
whose unique minimizer is $(1,1)$ and whose minimum value is zero. We take
$x_0=(-1,-1)$, $v_0=0$, $\gamma=1$, $\beta=0.5$, and $r=0.1$. The fine-step
and algorithmic step sizes are $5\times10^{-4}$ and $10^{-2}$, respectively,
and the final horizon is $T=40$. The continuous plots are restricted to
$0\leq t\leq10$, while the discrete plots show the first $1000$ iterates.
The computations themselves are continued up to the common physical
horizon $T=40$.

The behavior is different in this example. DF reaches
a permanent stop at $t=7.998$ for the fine-step dynamics and at $t=8.020$
for the algorithm. EH and IH do not permanently stop before $T=40$. Instead,
their gradient norms at $T=40$ are equal to $r$ up to numerical
precision. Figure~\ref{fig:example-two} shows the two corrected trajectories
approaching the boundary defined by $\norm{\nabla f}=r$. The limiting force does
not satisfy the strict interior condition. This example complements the
boundary example in Proposition~\ref{prop:boundary-no-stopping} and explains why the
strict inequality in \eqref{eq:strict-reference-margin} cannot simply be
replaced by a non-strict one.

\begin{figure}[t]
\centering
\includegraphics[width=0.98\linewidth]{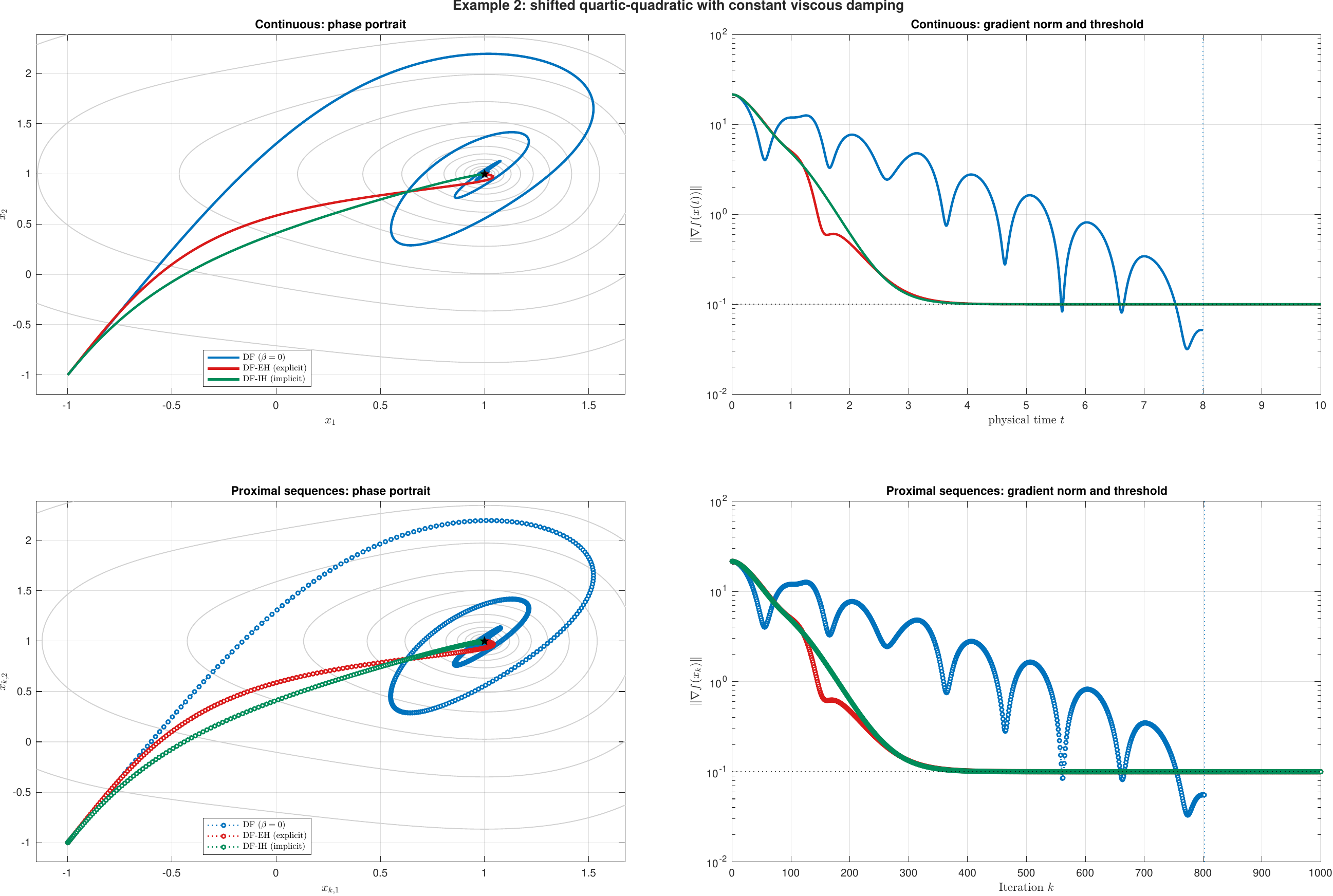}
\caption{Boundary case. The top row shows the continuous trajectories on
$0\leq t\leq10$. All three curves are solid. The bottom row shows the first
$1000$ iterates of the proximal algorithms; open circles identify every
computed iterate and dotted segments are only visual interpolants. DF enters
{the region $\norm{\nabla f}<r$} and stops permanently. EH and IH
approach the threshold $\norm{\nabla f}=r$. The computations were continued to
$T=40$ without a detected permanent stop for EH or IH.}
\label{fig:example-two}
\end{figure}
\FloatBarrier

Table~\ref{tab:numerical-stopping-times} collects the stopping times. The
agreement between the fine-step dynamics and the algorithms is particularly
close for the anisotropic quartic, while the coarser curved-valley
discretization correctly preserves the ordering and the damping gain. The
symbol $>40$ in {the boundary-case test} means absence of a detected
permanent stop before the prescribed horizon; it does not mean divergence.

\begin{table}[H]
\centering
\small
\setlength{\tabcolsep}{5pt}
\begin{tabular}{lcccccc}
\toprule
& \multicolumn{3}{c}{Fine-step dynamics} &
  \multicolumn{3}{c}{Proximal algorithms} \\
\cmidrule(lr){2-4}\cmidrule(lr){5-7}
Test & DF & EH & IH & DF & EH & IH \\
\midrule
Rotated quadratic & {--} & {--} & {--} & $7.350$ & $4.450$ & $4.450$ \\
{Curved valley} & {$7.986$} & {$6.020$} & {$6.068$}
& {$8.150$} & {$6.000$} & {$6.050$} \\
Anisotropic quartic & $12.237$ & $8.863$ & $8.808$ & $12.229$ & $8.863$ & $8.808$ \\
{Boundary case} & $7.998$ & $>40$ & $>40$ & $8.020$ & $>40$ & $>40$ \\
\bottomrule
\end{tabular}
\caption{Detected permanent stopping times in physical time. {A dash
indicates that no fine-step continuous simulation was performed for the
corresponding quadratic experiment.}}
\label{tab:numerical-stopping-times}
\end{table}
\FloatBarrier
{The modal experiment illustrates the exact quadratic stability
boundary. The quadratic trajectory gives the same discrete sequence for EH
and IH. In the nonlinear examples, the two Hessian-driven systems reduce the
oscillations in a similar way. The last example shows what may occur when the
limiting force lies on the boundary of $\partial\phi(0)$. We draw no
conclusion about the nonlinear stability boundary from these experiments.}

\section{Conclusion}
Starting from the implicit Hessian-driven model introduced in \cite{AlecsaLaszloPinta2021}, we studied its interaction with dry friction
when the viscous and shift coefficients are constant. This paper complements the previous works~\cite{AdlyAttouch2022,AdlyAttouch2020}. The shifted variable $x+\beta\dot x$ provides an energy suitable for both the continuous system and
its proximal discretization. This analysis gives finite-time stabilization under a strict interior condition, an $O(\beta^2)$ comparison with explicit
 Hessian-driven damping and finite convergence of the discrete iterations. The scalar boundary example shows, however, that
the equilibrium condition $-\nabla f(x_\infty)\in\partial\phi(0)$ does not necessarily imply finite-time stabilization.

Several questions remain open and needs further investigations. The first concerns the limit of the
static friction set. When $-\nabla f(x_\infty)\in\bd(\partial\phi(0))$, finite time
stabilization may hold or fail. A complete characterization of these two cases would be insightful. This would clarify 
whether the stopping times and the stopping positions of the explicit and the implicit systems can still be compared
without the strict interior condition.

The second question concerns the global Lipschitz condition of $\nabla f$ used in discrete analysis. This assumption is not
 necessary for the continuous time analysis, but it provides a fixed step size for the discrete energy estimation.
A local test or a backtracking rule may allow the step size to be chosen along the computed sequence. The difficulty lies in maintaining the discrete energy estimation and finite convergence when the step size varies.

The third question is mechanical. The potential $\phi(v)=r\norm{v}$ describes the simplified Coulomb friction law used here
because of its convex shape and its direct relationship with proximal algorithms in optimization.
For unilateral contact, one must also consider the normal reaction, a state-dependent friction threshold, and possibly 
separate static and dynamic friction coefficients. It remains to be determined whether the displaced energy and the 
finite-time arguments extend to this setting.

All these questions are beyond the scope of this manuscript and will be investigated in a future research project.

% ============================================================================
\appendix

\section{Proof of the Sharpness Result}
\label{app:sharpness-proof}

\begin{proof}
Take
$$
f(x)=\frac{x^2}{2}+a(1-\cos x),
\qquad 0<a<1,
$$
$\phi(v)=r|v|$, $\bar x_0=\pi/2$, and $v_0>0$.
The notation $\bar x_0$ is used here for the prescribed
initial position, in order to distinguish it from the reference trajectory
$x^0$. Then
$$
f''(x)=1+a\cos x\geq1-a>0,
\qquad
f'''(x)=-a\sin x,
$$
so all the announced global regularity properties hold and
$f'''(\bar x_0)=-a\neq0$.

We first work with smooth fixed-sign extensions of the dynamics. As long as
the velocity is positive, {the dry-friction selection is $r$}, and the two models
reduce to
\begin{align*}
\dot x_\beta^I&=v_\beta^I,
&\dot v_\beta^I
&=-\gamma v_\beta^I-r-f'(x_\beta^I+\beta v_\beta^I),\\
\dot x_\beta^E&=v_\beta^E,
&\dot v_\beta^E
&=-\gamma v_\beta^E-r-f'(x_\beta^E)
-\beta f''(x_\beta^E)v_\beta^E.
\end{align*}
We regard these smooth equations as extensions beyond their first velocity
zeros. For $J\in\{I,E\}$, denote the corresponding smooth
phase-space vector field by $F_J(\beta,z)$, $z=(x,v)$. Since the function
$f$ chosen above is $C^\infty$, each $F_J$ is $C^2$ jointly in $(\beta,z)$
on bounded sets. The parameter-dependent Cauchy theorem for ordinary
differential equations therefore shows that, on every common compact
existence interval, the map
\[
\beta\longmapsto z_\beta^J=(x_\beta^J,v_\beta^J)
\]
is $C^2$ with values in $C^1$ of that interval. At $\beta=0$ the two vector
fields coincide and
\[
\partial_\beta F_I(0,x,v)
=\partial_\beta F_E(0,x,v)
=\bigl(0,-f''(x)v\bigr).
\]
The first parameter variations consequently solve the same linear Cauchy
problem with zero initial data and hence coincide. Define
\[
(U,W):=\frac12\left.
\partial_{\beta\beta}\bigl(z_\beta^I-z_\beta^E\bigr)
\right|_{\beta=0}.
\]
Taylor's formula in $C^1$ then gives, uniformly on the compact interval,
\begin{equation}
\label{eq:dev-var-seconde-unif}
x_\beta^I-x_\beta^E=\beta^2U+o(\beta^2),
\qquad
v_\beta^I-v_\beta^E=\beta^2W+o(\beta^2),
\end{equation}
where subtraction of the second variational equations yields
\begin{equation}
\label{eq:sys-var-optimalite}
\dot U=W,
\qquad
\dot W+\gamma W+f''(x^0)U
=-\frac12f'''(x^0)(v^0)^2,
\qquad U(0)=W(0)=0.
\end{equation}
In particular,
$$
\dot W(0)=-\frac12f'''(\bar x_0)v_0^2
=\frac a2v_0^2>0.
$$

We now choose the friction radius in a way that controls both the stopping
event and the sign of $W$ at that event. Write $(x_r^0,v_r^0)$ for the
reference fixed-sign solution and rescale time by
$$
X_r(s)=x_r^0(s/r),
\qquad
V_r(s)=v_r^0(s/r).
$$
Up to the first zero of the velocity,
$$
X_r'=\frac1rV_r,
\qquad
V_r'=-1-\frac1r\bigl(\gamma V_r+f'(X_r)\bigr).
$$
Continuous dependence on $1/r$ gives, uniformly on compact $s$-intervals,
$$
X_r(s)\longrightarrow \bar x_0,
\qquad
V_r(s)\longrightarrow v_0-s.
$$
The limiting zero at $s=v_0$ is transverse. It follows that, for all
sufficiently large $r$, $v_r^0$ has a unique first zero $T_r^0$ near
$v_0/r$, is positive on $[0,T_r^0)$, and
$$
rT_r^0\longrightarrow v_0.
$$
Moreover, $X_r(rT_r^0)\to \bar x_0$, so
$|f'(x_r^0(T_r^0))|<r$ for all sufficiently large $r$. This zero is
therefore a permanent stop with a strict friction margin, and
$$
\dot v_r^0(T_r^0-)=-r-f'(x_r^0(T_r^0))<0.
$$

Let $(U_r,W_r)$ solve \eqref{eq:sys-var-optimalite} along this
reference trajectory. On the rescaled interval set
$$
\widetilde U_r(s)=U_r(s/r),
\qquad
Z_r(s)=rW_r(s/r).
$$
Then
\begin{align*}
\widetilde U_r'(s)&=\frac{1}{r^2}Z_r(s),\\
Z_r'(s)&=-\frac\gamma r Z_r(s)
-f''(X_r(s))\widetilde U_r(s)
-\frac12 f'''(X_r(s))V_r(s)^2.
\end{align*}
The preceding compact convergence and Gronwall's inequality imply
$\widetilde U_r\to0$ and
$$
Z_r(s)\longrightarrow
-\frac12f'''(\bar x_0)
\int_0^s(v_0-\sigma)^2\,d\sigma
$$
uniformly for $s$ in a fixed neighborhood of $[0,v_0]$. Evaluating at
$s=rT_r^0$ gives
$$
rW_r(T_r^0)
\longrightarrow
\frac a2\int_0^{v_0}(v_0-s)^2\,ds
=\frac{av_0^3}{6}>0.
$$
Fix from now on one sufficiently large $r$ for which the strict stopping
condition holds and $W_r(T_r^0)>0$, and write $T^0=T_r^0$,
$(x^0,v^0)=(x_r^0,v_r^0)$, and $(U,W)=(U_r,W_r)$.

The transverse zero and the strict terminal margin persist
for both smooth extensions when $\beta$ is small. Indeed, set
$G_J(\beta,t)=v_\beta^J(t)$. The $C^2$ parameter dependence established
above gives $G_J\in C^2$ near $(0,T^0)$, while
\[
G_J(0,T^0)=0,
\qquad
\partial_tG_J(0,T^0)=\dot v^0(T^0-)\neq0.
\]
The implicit-function theorem therefore supplies a unique $C^2$ function
$T_\beta^J$ near $\beta=0$ such that
$G_J(\beta,T_\beta^J)=0$. Uniform continuous dependence, positivity of
$v^0$ on compact subintervals of $[0,T^0)$, and transversality at $T^0$
show that this root is the first velocity zero. At either zero one has, for
small $\beta$,
$$
|f'(x_\beta^J(T_\beta^J))|<r.
$$
Thus the smooth fixed-sign branch coincides with the physical dry-friction
trajectory before the zero, and the exact sticking criterion makes this
zero its permanent stopping time.

Differentiating
$G_J(\beta,T_\beta^J)=0$ at $\beta=0$ shows that the first derivatives of
$T_\beta^I$ and $T_\beta^E$ coincide, because the first parameter
variations of the velocity branches coincide. Since the two root functions
are $C^2$,
\[
T_\beta^I-T_\beta^E=O(\beta^2).
\]
We now evaluate the velocity difference at the moving roots. The uniform
expansion \eqref{eq:dev-var-seconde-unif}, continuity of $W$,
and the preceding estimate give
\[
v_\beta^I(T_\beta^I)-v_\beta^E(T_\beta^I)
=\beta^2W(T^0)+o(\beta^2).
\]
Moreover, Taylor's formula in time, uniformly for small $\beta$, gives
\[
v_\beta^E(T_\beta^I)-v_\beta^E(T_\beta^E)
=\dot v^0(T^0-)(T_\beta^I-T_\beta^E)+o(\beta^2).
\]
Adding these identities and using
$v_\beta^I(T_\beta^I)=v_\beta^E(T_\beta^E)=0$ yields
$$
0=\beta^2W(T^0)
+\dot v^0(T^0-)(T_\beta^I-T_\beta^E)+o(\beta^2).
$$
Consequently,
$$
\lim_{\beta\downarrow0}
\frac{T_\beta^I-T_\beta^E}{\beta^2}
=-\frac{W(T^0)}{\dot v^0(T^0-)}>0,
$$
which proves \eqref{eq:stopping-time-optimality}.

Finally, after $r$ has been fixed, the inequality $\dot W(0)>0$ allows us
to choose $t_*\in(0,T^0)$ sufficiently small that $W(t_*)>0$. The uniform
expansion \eqref{eq:dev-var-seconde-unif} then gives
\eqref{eq:state-optimality} with $c_*=W(t_*)>0$.

\end{proof}

\section{Algebraic Verification of the Discrete Lyapunov Matrix}

For completeness, the terms obtained after testing
\eqref{eq:discrete-selected-equation} can be grouped before any coefficient
is simplified:
\begin{align*}
0\geq{}&
\E_{k+1}^h-\E_k^h+h\phi(v_{k+1})
+h\gamma\norm{v_{k+1}}^2\\
&+\left(\frac12+\frac\beta h+\frac{\beta\gamma}{2}\right)
\norm{d_{k+1}}^2\\
&-\frac L2\left(
h^2\norm{v_{k+1}}^2
+2h\beta\ip{v_{k+1}}{d_{k+1}}
+\beta^2\norm{d_{k+1}}^2
\right).
\end{align*}
The coefficients of $\norm{v_{k+1}}^2$, $\norm{d_{k+1}}^2$, and
$\ip{v_{k+1}}{d_{k+1}}$ are therefore $A_h$, $C_h$, and $-Lh\beta$,
respectively. In particular, the off-diagonal entries of the representing
matrix are $-Lh\beta/2$.

Expanding its determinant gives
\begin{align*}
\det M_h
&=\left(h\gamma-\frac{Lh^2}{2}\right)
\left(\frac12+\frac\beta h+\frac{\beta\gamma}{2}
-\frac{L\beta^2}{2}\right)
-\frac{L^2h^2\beta^2}{4}\\
&=\beta\gamma
+\frac{1+\beta\gamma}{2}(\gamma-L\beta)h
-\frac{L(1+\beta\gamma)}{4}h^2,
\end{align*}
which is exactly \eqref{eq:lyapunov-determinant}. At
$h=h_{\mathrm{Lyap}}$ the determinant vanishes. The separate verification
of $A_h\geq0$ in Theorem~\ref{thm:lyapunov-certificate} is what makes this
boundary positive semidefinite rather than merely singular.

% ============================================================================

% ============================================================================
\section*{Acknowledgements}
The author gratefully acknowledges support from the Math AmSud project
N$^{\circ}$51756TF (VIPS), ECOS Project C24E06 and the FMJH Gaspard Monge
Program for Optimization and Data Science.

\end{document}